\documentclass{article}

\usepackage[english]{babel}
\usepackage[utf8x]{inputenc}
\usepackage[margin=2.5cm,a4paper]{geometry}
\usepackage{amsmath,tikz}
\usepackage{etoolbox}
\usetikzlibrary{matrix}
\usetikzlibrary{positioning}
\usepackage{amssymb}
\usepackage{amsthm}
\usepackage{moresize}
\usepackage{graphicx}
\usepackage[colorinlistoftodos]{todonotes}
\usepackage{mathtools}
\usepackage{hyperref}
\usepackage{cite}
\usepackage{fancyhdr}
\usepackage{float}
\usepackage{lscape}
\usepackage{enumerate}
\usepackage{listing}
\usepackage{graphics}
\usepackage{bbm}
\usepackage{wrapfig}
\usepackage{framed}
\usepackage{array}
\usepackage{listings}
\usepackage{hhline}
\usepackage{algorithm}
\usepackage{algpseudocode}
\usepackage{algorithmicx}
\usepackage{breqn}
\usepackage{verbatim}
\usepackage[toc,page]{appendix}
\usepackage{caption}
\usepackage{color}
\usepackage[normalem]{ulem}
\usepackage{subcaption}
\usepackage{nomencl}
\usepackage{multirow}
\usepackage{ wasysym } 
\usepackage{leftidx}
\usepackage{times} 
\usepackage{hyphenat}
\usepackage{ragged2e}
\usepackage{blindtext}
\usepackage{longtable}
\usepackage{setspace}
\usepackage[autostyle]{csquotes} 
\usepackage{listings}
\usepackage{longtable}
\usepackage{rotating}

\numberwithin{equation}{section}
\addto{\captionsenglish}{}

\DeclareMathOperator{\prob}{{\mathbb{P}}}
\DeclareMathOperator{\var}{{\text{Var}}}
\DeclareMathOperator{\expec}{{\mathbb{E}}}

\DeclareMathOperator{\e}{{e}}

\newcommand{\Q}{Q}
\newcommand{\M}{M}
\newcommand{\Y}{Y_{y,c,\lambda,1}}
\newcommand{\Yl}[1]{Y_{#1,\nu\expec[X]/K,\lambda,1}}

\newcommand\1{\mathbbm{1}}

\newtheorem{theorem}{Theorem}[section]
\newtheorem{proposition}[theorem]{Proposition}
\newtheorem{lemma}[theorem]{Lemma}

\newtheorem{remark}[theorem]{Remark}

\definecolor{halfgray}{gray}{0.55} 
\definecolor{purplish}{rgb}{0.41, 0.41, 0.64}
\definecolor{navy}{rgb}{0,0,0.52}
\definecolor{webbrown}{rgb}{.6,0,0}
\definecolor{Maroon}{cmyk}{0, 0.87, 0.68, 0.32}
\definecolor{NiceRed}{rgb}{0.41,0, 0}
\definecolor{Black}{cmyk}{0, 0, 0, 0}
\definecolor{Peacefulbackground}{rgb}{0.933333,0.9098,0.9098}
\definecolor{Poop}{RGB}{165,82,35}
\definecolor{darkgreen}{RGB}{50,205,50}

\hypersetup{
	colorlinks=true, 
	urlcolor=webbrown, linkcolor=NiceRed, citecolor=purplish, 
}

\title{A Holistic Approach for Bitcoin Confirmation Times \& Optimal Fee Selection }
\author{Rowel C. Gündlach \and Ivo V. Stoepker \and Stella Kapodistria \and Jacques A. C. Resing}
\date{\today}

\begin{document}

\maketitle
\begin{abstract}

Bitcoin is currently subject to a significant pay-for-speed trade-off. This is caused by lengthy and highly variable transaction confirmation times, especially during times of congestion.  Users can reduce their transaction confirmation times by increasing their transaction fee.
In this paper,  based on the inner workings of Bitcoin, we propose a model-based approach (based on the Cram\'er-Lundberg model) that can be used to determine the optimal fee, via, for example, the mean or quantiles, and models accurately the confirmation time distribution for a given fee. 
The proposed model is highly suitable as it arises as the limiting model for the mempool process (that tracks the unconfirmed transactions), which we rigorously show via a fluid limit and we extend this to the diffusion limit (an approximation of the Cram\'er-Lundberg model for fast computations in highly congested instances). 
We also  propose methods (incorporating the real-time data) to estimate the model parameters, thereby combining model and data-driven approaches. 
The model-based approach is validated on real-world data and the resulting transaction fees  outperform, in most instances, the data-driven ones. 
\end{abstract}

\tableofcontents

\section{Introduction}
When transferring money, we are accustomed to  quick transactions with clear pre-specified rules set by central third parties. However, with cryptocurrencies that is not the case. For example, in Bitcoin (the most popular example of a cryptocurrency), upon placing a transaction, there is a random amount of transaction delay on which the user has limited view and control. This is especially hindering at times when there is a peak in transactions demand.  E.g., at the end of 2017, due to a surge in popularity, users experienced long delays in their transactions. Even during times when the system is less busy, one needs to take a significant amount of delay into account.

A key performance metric for the transaction throughput is the so-called  \textit{confirmation time} (also referred to as transaction confirmation time or confirmation delay). The confirmation time can be viewed (in a stylised framework) as the time between sending a transaction (into the mempool: the pool of pending transactions that are waiting to be included in a block) and the moment the transaction is included in a block on the blockchain\footnote{In practice,  when making a Bitcoin transaction, recipients usually require somewhere between 2 and 6 confirmations to consider the transaction as valid. However, we focus here on the  first confirmation, as once the transaction is included in a Bitcoin block and thus obtains the first confirmation, the additional confirmations will follow and it will need a further  10 minutes  (on average) for each additional confirmation. In the Bitcoin literature, the term \lq confirmation time\rq, is an ambiguous term. We refer to Remark \ref{rem:cnofirmationtime_def} for a discussion.}.

Users have some degree of control on their transactions' confirmation time by selecting a fee they can attach to a transaction akin to classical pay-for-speed settings. Namely, the user can effectively decrease the confirmation time by including a higher fee to the transaction (for the miner, i.e. the user that decides which transactions are put in a block). In the analysis that follows, we fix the fee\footnote{Technically the \textit{fee density}, i.e. the fee per byte ratio determines how profitable a transaction is for a miner. For convenience we simply refer to this as the fee.} and analyse, given the fee, the (stochastic) properties of the confirmation time. 

Considering a target transaction with a certain fee, then its confirmation time can be determined via the so called mempool. 
The mempool keeps track of all the transactions (number of transactions, their weight  and fee per transaction) that have not been confirmed yet and can be restricted to specifically those transactions that have a higher fee than the target transaction. We assume that such transactions arrive with an almost constant high rate and are small in weight with respect to the size of a block. As a result the arrival process of transactions seems almost linear. Therefore, the cumulative weight of transactions with a higher fee than the target transaction is modelled by a process that 
\begin{itemize}
    \item starts at a certain height corresponding to the cumulative weight of transactions that have a higher fee;
    \item  
increases almost linearly due to new transactions arriving with higher fee;
\item jumps down when a block is found, thereby confirming the transactions with the highest fee,
\end{itemize} c.f. Figure \ref{fig:mempool_snap}. 
This process bears a close resemblance to the Cram\'er-Lundberg  (CL) model \cite[Chapter IV]{asmussen2020}, see Section \ref{sec:model_theory} for details.
The target transaction is confirmed when this process crosses 0 for the first time, i.e. a block is found and the cumulative weight of transactions with a higher fee is smaller than the block size.
Therefore, the confirmation time is modelled as the time to ruin in the CL model (first time to hit 0). 

\begin{figure}
    \centering
    \includegraphics[scale=0.25]{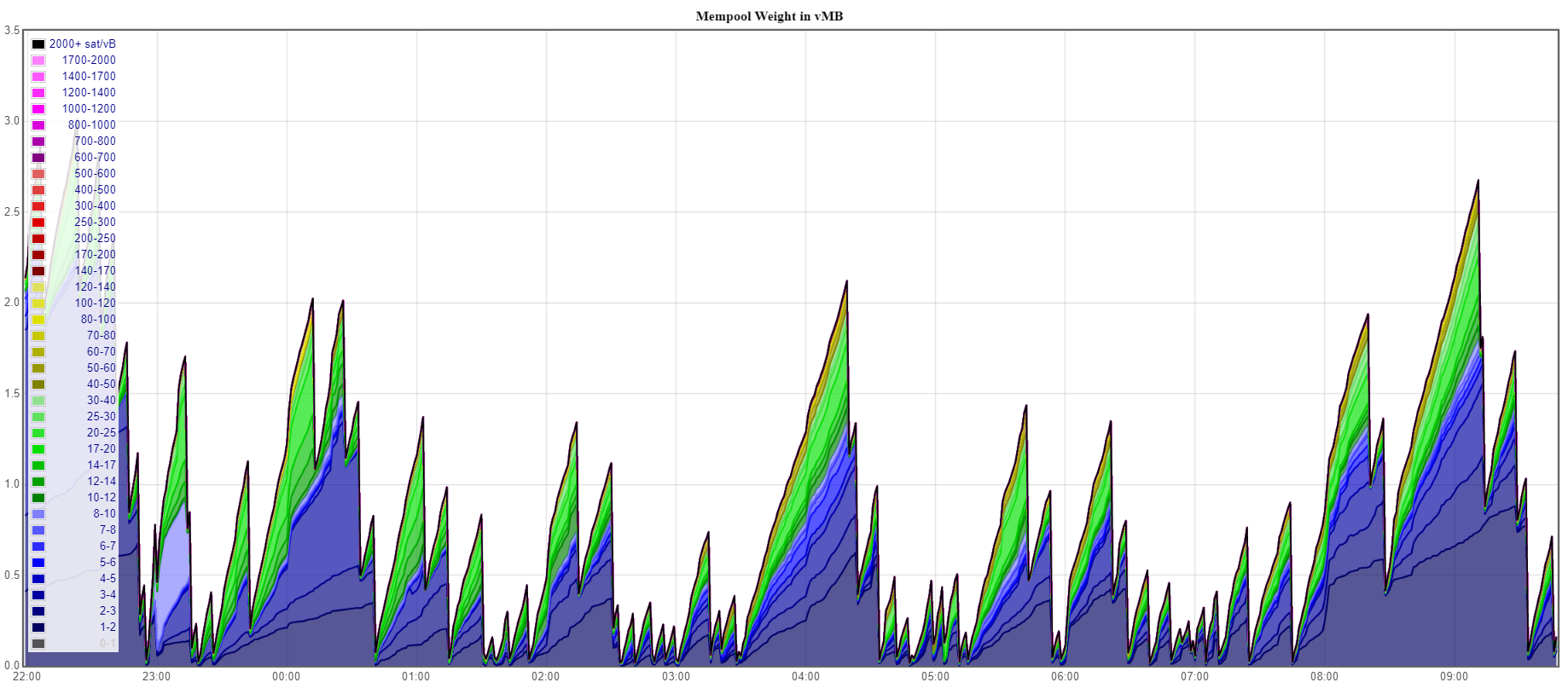}
    \includegraphics[scale=0.35]{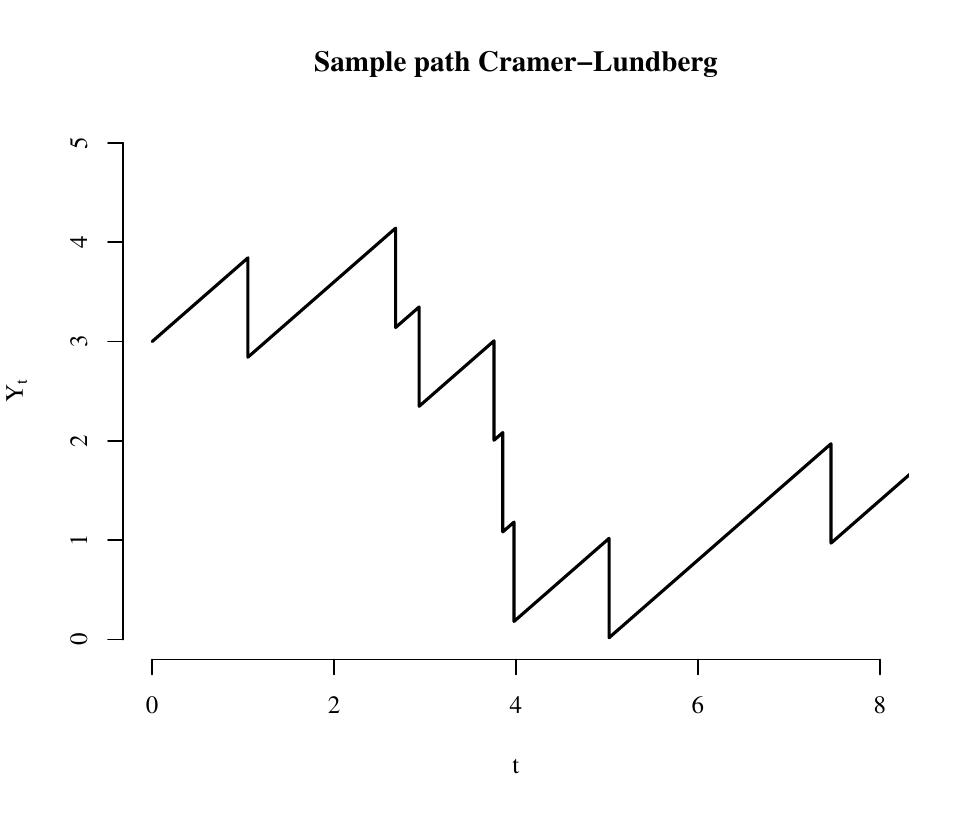}
    \caption{
    Comparison of a real-world realisation of the mempool versus a CL model.
    On the left, a snapshot of the mempool taken on 11-4-2020 from 00:00 am to 8:00 am. The cumulative mempool weight (vMB) is outlined on the $y$-axis over time on the $x$-axis. Source: \url{https://mempool.jhoenicke.de}.
    On the right, a sample path of the CL model. }
    \label{fig:mempool_snap}
\end{figure}

In what follows, we propose a holistic overview of this model's position in the current Bitcoin confirmation time literature, historical data and its practical application. We divide the contributions of this work in these four parts:
\begin{description}
\item[\textbf{Relation of the model to the established literature.}]
We provide a general model description of established models in the literature, most of which are special cases of the \textit{batch-service queue} (BSQ), see Section \ref{sec:queueing_model} for a formal definition. We show that under an appropriate scaling (that we argue is realistic in practice) the established models have a CL fluid limit and a Brownian motion (BM) diffusion limit when the system is in a critical region. The CL model and BM can therefore be seen as overarching mathematical models for transaction confirmation times.

\item[\textbf{Model-based decision making.}]
As the CL model and the BM have been extensively studied in the literature, we show that  established results can be tailored towards determining the optimal fee in an elegant and simple manner.
For a given fee, we determine the confirmation time density  and its mean through a simple recursive scheme, see Section 2. Users can make an optimal fee decision based on these metrics. For the fluid limit, these results extend to both measurements in time units  (continuous) and in number of blocks (discrete).

While the mean and distribution of the confirmation time under a BM has already been studied in detail in \cite{Kooops2018} and \cite{Gundlach2021}, our work rigorously bridges the aforementioned results to established models from the literature via the diffusion limit.

\item[\textbf{Validation study.}]
We formally validate the choice of the CL model using real time data and demonstrate the practical application of the proposed mathematical model. We discuss in detail how to validate the model and the results, and show that the model can be used in practice to estimate confirmation times, as well as for deciding on the optimal fee.
As this validation study covers a data set that extends over a long period of time (two years), we also demonstrate its robustness.

\item[\textbf{Practical application and prediction.}]
We provide methods to obtain the model parameters from real-world data. We then use the proposed model to derive the optimal fee for a target confirmation time. The results are compared to those obtained by data-driven methods and we demonstrate that the proposed approach brings significant benefits to the customers allowing them to more accurately choose the right fee.

\end{description}

\noindent
\textit{Organisation of this section.}
In what follows, we provide some details on the aforementioned main contributions. In Section \ref{sec:relation to bc lit} we discuss how the model fits in the current literature and how it relates to the CL model. Moreover, we briefly discuss some of the benefits of having the CL both from a modelling and practical perspective. Then in Section \ref{subsec:1modelval} we provide more details on how we validate the model with real-world data. In Section \ref{subsec:1modelperf}, we discuss how the model can be used in practice and how it performs in determining the optimal fee. We then discuss methods and applications beyond stochastic modelling of Bitcoin confirmation times in Section \ref{sec:beyond stoch model of bc ct}. Section \ref{sec:paper_overview} concludes with an overview of the rest of the paper.

\subsection{Relation and benefits of the model to the established literature}
\label{sec:relation to bc lit}

The majority of stochastic models in Bitcoin confirmation times are based on queueing models, see \cite{Kasahara2021} and the references therein. In such models, transactions arrive according to a Poisson process and are confirmed in \textit{batches}, meaning that after a certain random amount of time the first (at most) $K$ transactions are confirmed, regardless of the number of transactions in the queue. This is an example of a \textit{batch-service queue} (BSQ), where customers represent transactions, and a batch-service coincides with the mining of a block with many transactions being confirmed at once. For more details on the BSQ see for example \cite[Chapter 4]{Chaudhry1983}. 

To the best of our knowledge,
\cite{Kasahara2019} were the first to propose a link between the bitcoin mempool and the BSQ. To cater to different fee levels,
the authors take different priority classes into account and derive an expression for the expected confirmation time per class.
The BSQ (without priorities), is also discussed in \cite{Kawase2017, Kawase2020}. Here, the batch confirmations are formulated using the so-called gated service discipline, i.e. upon arrival a transaction cannot be included in the next block. For this model, the expected confirmation time is derived. 
\cite{Li2018} also considers a BSQ and incorporates the effect of propagation delay on top of the time between blocks that corresponds to a service time.   This work is later extended in \cite{Li2019}, where the transaction arrival process, service process and service discipline are generalised. This work was further extended in \cite{Malakhov2023}, where the authors go beyond the steady-state solution and find the expected delay in a transient setting with initial backlog.

However, these models have some pressing shortcomings. 
Due to the complexity of the BSQ model, analytic results are only available for the mean confirmation time and are given as involved expressions that are difficult to numerically evaluate.
Moreover, it is not clear whether such models realistically capture the real-world application. 
Finally, and most commonly in the BSQ literature, the model keeps track of the number of transactions and further assumes that the block size is deterministic. However, as every transaction has a different weight and blocks have a fixed maximum capacity of 1MB, these two modelling choices (number of transactions and deterministic block sizes) do not capture the full reality of the transaction confirmation process.
In the following we show that the model proposed in this paper can overcome all aforementioned shortcomings.

Contrary to the assumptions in the BSQ literature, in reality, the block (or batch) size  is not expressed in the number of transactions (or customers) but rather with its weight (or workload) of 1MB. Therefore, one needs to adopt one of the following points of view:
\begin{enumerate}
    \item Keeping track of the number of transactions in the system. Then every arrival brings 1 transaction, and the block size (expressed in the number of transactions) is random. 
    Under this modelling choice, it is important to note that these random block sizes depend on the arrival process. We refer to this as the \textit{transaction count process}. 
    \item Keeping track of the transaction weights in the system. Then every arrival brings a random amount of weight (or workload) to the system. These weights follow a continuous distribution, which implies that a block can never be exactly filled. This leads to random block sizes that are also dependent on the arrival process. We refer to this as the \textit{cumulative weight process}. 
\end{enumerate}

Despite these two seemingly different modelling choices, it is remarkable that in the limiting regime, i.e. when we take the limit of the mean number of arrivals and the block sizes proportionally to go to infinity, it turns out that under mild conditions, the Bitcoin mempool as a BSQ under either point of view converges to the same (fluid) limiting process, which is the CL model. This connects the BSQ and CL model in a natural way, where the CL model can be used as an alternative to overcome the modelling dependency issues under the BSQ.

Real-world data shows that the mempool is consistent with the dynamics of the (fluid) limiting process and it behaves according to a CL model, see for example Figure \ref{fig:mempool_snap}. The main contributing reasons are summarised as follows:
Firstly, data shows that around 5 transactions arrive in the mempool per second.
which corresponds to around $3\cdot 10^3$ transactions per 10 minutes. In contrast, block arrivals happen roughly every 10 minutes.
Secondly,
block sizes are fixed at $10^6$B, whilst a single transaction  is (generally) significantly smaller (200-300B). So taking the limit of the arrival rate and the block sizes to infinity provides a suitable approximation, making the BSQ fluid limit an appropriate model to use for practical applications. This is in line with the so called snap-shot principle mentioned in \cite{DAuria2005}, a consequence of scaling the space more heavily than scaling time.

The application of the CL model for Bitcoin transactions has been discussed in the literature before. 
\cite{Kooops2018} was the first to 
link the mempool to a CL model.
However, this analysis was limited to an approximation by its heavy-traffic counterpart resulting in the confirmation time being approximated with an inverse Gaussian distribution. This mathematical framework was later refined both theoretically in \cite{Gundlach2021}, and practically via a robustness study in \cite{Stoepker2021} where it was shown that confirmation times are robust to small deviations of the assumptions of the CL model. 
 
We provide also the necessary bridge between the BSQ and BM as follows. We jointly scale the rate of incoming transactions and the block size of the BSQ to infinity, such that the resulting limiting process is in a heavy-traffic regime.
Under this scaling we find that the BSQ converges to the BM.
This links the aforementioned work of \cite{Kooops2018,Gundlach2021} directly to the BSQ, instead of via the CL
and therefore, simplistic expressions for the confirmation times can be used.

However, the heavy-traffic approximation is only suitable in practice when the rate of incoming transactions is sufficiently high. While the approximation works quite well even outside of the heavy-traffic regime, c.f. Section \ref{sec:results}, in light-traffic cases one needs to resort to deriving the confirmation time distribution of the CL which is not straightforward. We present an elegant solution to this in Section \ref{sec:model_theory} by introducing simplified recursive expressions that determine the confirmation time density and the mean confirmation time, both measured in time (units) and number of blocks.

\subsection{Validation study}
\label{subsec:1modelval}
As briefly mentioned above, one of the inspirations for modelling the mempool data with the CL model comes  from the visual similarities of a data path of the mempool to a sample path of the CL model. See Figure \ref{fig:mempool_snap} for a snapshot of the mempool data and notice its similarity to the CL model.
The colours capture the different fee levels $\phi$.
While informal, this visual similarity indicates that the mempool roughly satisfies the same properties as the proposed CL model.

In this work, we provide a formal comparison, 
by deriving the distribution of the confirmation time based on the CL model and show that the data fits the distribution.
The results of this case study show promise, see for example Figure \ref{fig:Densities_4x1}, where  we provide an insight in the confirmation time distribution of the theoretical model and overlay the empirical distribution of the collected data.  
Details and further results on, for example, a comparison of the mean confirmation times are provided in Section \ref{sec:case_study}.

\begin{figure}[h!]
    \includegraphics[scale=0.8]{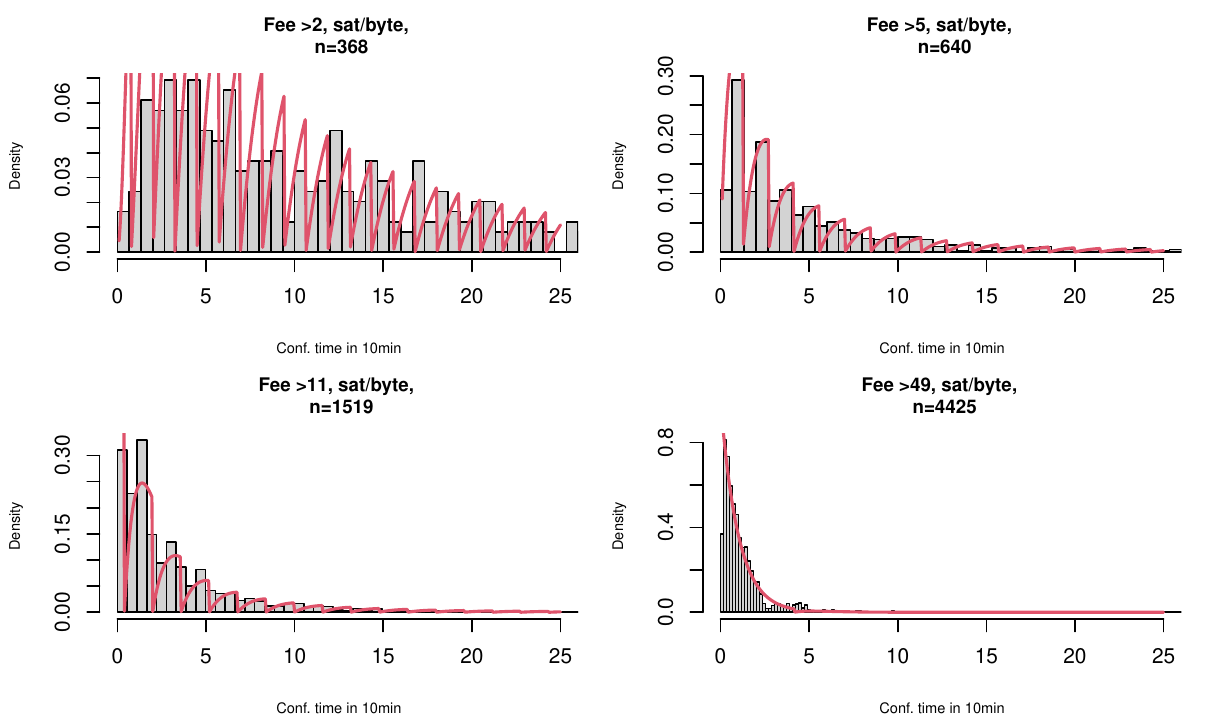}
    \caption{Histograms of $n$ extracted confirmation times for  $\phi\in\{3,6,12,50\}$ sat/byte (from left to right). We overlay (plotted in red) the density 
    of the  confirmation times
    to the theoretical distribution in Equation \eqref{eq:densityut}.}
    \label{fig:Densities_4x1}
\end{figure}

\subsection{Practical application and prediction}
\label{subsec:1modelperf}
Although the CL model seems to fit the mempool data exceptionally well, it comes with a caveat. While Section \ref{sec:model_theory} presents a recursive scheme on how one can compute the distribution or the mean of the confirmation times, such a scheme may still
require significant computational power. 
To overcome this challenge, we turn to the BM,
which can be obtained as the heavy-traffic limit of the CL model,  c.f. Remark \ref{rem:diffusionapprox}. We show that the BM in such instances is still a very accurate model and it offers a simple closed-form expression for the distribution and the mean of the confirmation times. This is no surprise as, in the heavy-traffic regime, the time to ruin of the CL model converges to the time it takes for the BM to hit 0, see Theorem \ref{thm:diffusion_limit}. The distribution of the time it takes the BM to hit 0 is well known and it is distributed according to the inverse Gaussian (IG) distribution, which is fast to evaluate. 

Table \ref{tab:E(T)results} in Section \ref{sec:results} indicates that the BM approximation is close to the CL model in terms of the mean confirmation time. 
In Table \ref{tab:CS1} in Section \ref{sec:case_study},  we provide an overview of  mean empirical confirmation times obtained from the data that are closely approximated by the mean time for a BM to hit 0.

An additional benefit of the BM approximation, is that it allows for optimal fee estimation based on realised transaction confirmation times, which is less data intensive than mempool data. In Section \ref{sec:cs_performance}, we use the BM approximation to determine, for a given target confirmation time, the optimal fee. In 46\% of the cases, the predicted fee is indeed the optimal fee. In contrast, the purely data-driven model is accurate only in 9\% of the cases.

\subsection{Beyond the stochastic modelling of Bitcoin confirmation times}
\label{sec:beyond stoch model of bc ct}
We finally step outside the field of the stochastic models used to model Bitcoin confirmation times, to discuss related work and applications. We first discuss some other methods such as statistics and machine learning methods that were proposed to determine the optimal fee. Then, we highlight different fields of applications where the model can also be applied.

The literature on Bitcoin confirmation times extends beyond the BSQ and CL model, which we illustrate as follows. 
Other stochastic models that have been explored are a game theoretical queueing model \cite{Li2020} or a simulation based model \cite{Memon2019}. 
From a related field, many data-driven methods have been introduced over the years. While some combine these methods with a mathematical model
\cite{Ricci2019}, others rely solely on machine learning tactics. 
\cite{Zhang2022} analyses the confirmation time falling into a specific time interval, thereby translating the problem to a classification problem. They then use multiple established machine learning techniques to predict the confirmation time.
\cite{Tedeschi2019} uses different neural networks to predict if a transaction will be included in the next block. 
\cite{abdullah2018} and \cite{Singh2020} discuss supervised learning algorithms for confirmation times in Bitcoin and  Ethereum respectively. \cite{Gebraselase2021} performs an elaborate time series analysis for exploration and prediction of various properties of Bitcoin, one of which is the confirmation time.

While the paper is tailored to Bitcoin confirmation times, its applications reach much further. It was noticed before in \cite{Shang2023} that the fundamental problem of paying a transaction fee for faster transaction confirmation is related to operations management pricing. More specifically, the model portrays a modern application of \textit{pay-for-speed} -  the more you pay the shorter the expected service time. Such protocols are omnipresent in our society. Notable instances are the costlier same-day or express delivery of delivery companies, fast-passes at amusement parks or priority boarding of planes. 
However, in these examples the price options are set and cannot be chosen by the user. This work therefore extends to \textit{pay-what-you-want} - the more you \textit{decide} to pay the shorter the expected service time. Such applications are more scarce in society, but are on the rise with companies as DoorDash and Instacart, a fast-food and grocery delivery service respectively. The delivery personnel sees orders of customers, which they can decide to fulfil for some fixed reward plus an optional tip customers can add on their order. It goes without saying that personnel chooses orders with the highest tip-to-fulfilment ratio first. 
We refer to \cite{Shang2023} for other examples  of pay-what-you-want applications in practice.
In such examples there is similar interest to pay the lowest fee or cost possible to receive a good or service within a certain time frame for which the model can be used.

A key contribution the paper provides is the fluid and diffusion limit of the batch service queue with many arrivals and a large batch size. The model therefore extends also to different applications of the BSQ, see for example \cite{Deb1973,Sasikala2016}. This application specifically extends to batch service queues where the maximum batch size is not measured in number of customers, but measured in some different quantity. This is especially relevant when goods cannot be broken down. Relevant applications extend to vertical transport systems (where complex machinery as forklifts, cranes or elevators is capable of transporting a maximal weight rather than a maximal number of elements), packing of transport vehicles (where delivery trucks are limited by the volume of storage capacity or weight) or road traffic (where a green light serves cars based on their speed and size).

\subsection{Paper overview}
\label{sec:paper_overview}
The rest of this paper is organised as follows: We introduce the CL model, the primary model of this work, and the needed notation in Section \ref{sec:model_theory}. In this section, we also provide a short overview of the theoretical results for the confirmation time distribution and mean. We then contrast the CL model with the BSQ that we introduce in Section \ref{sec:queueing_model}. We also formalise the relation with the BSQ with its fluid and diffusion limit. We then provide detailed proofs of these claims:
Section \ref{sec:math_analysis} contains the proofs of the results from Section \ref{sec:model_theory} and Section \ref{sec:fluidlimit proof} contains the proofs of the results from Section \ref{sec:queueing_model}.
We present and briefly discuss some numerical results in Section \ref{sec:results}. In Section \ref{sec:case_study}, we present a case study, where we demonstrate how well the proposed model fits to real-world Bitcoin data  and how well it performs in comparison to purely data-driven methods. We conclude the paper in Section \ref{sec:conclusion}.

\section{The Cram\'er-Lundberg (CL) model}
\label{sec:model_theory}
In what follows, we consider  a transaction with a fixed fee  $\phi$. We assume that, when the cumulative  weight of transactions in the mempool with higher fee than $\phi$ is smaller than the block size, then if a block is found, the transaction with fee level $\phi$ is included in the found block (and thereby confirmed). 
This corresponds with the assumption that block miners are rational and confirm transactions with the highest fee first.
Therefore, when modelling the cumulative mempool weight, we only include transactions with higher fee than $\phi$.

The cumulative mempool weight is modelled by a stochastic process $Y_{y_\phi,c_\phi,\lambda,B}$, that starts with some initial load $y_\phi$. The initial load represents the initial transactions of higher fee than $\phi$ that were already in the mempool at time $t=0$. Over time, the mempool grows linearly  with slope $c_\phi$. The linear increase represents incoming transactions of higher fee than $\phi$. At random instances, a block is mined. This is modelled with a Poisson process $Z_\phi(t)$ with rate $\lambda$. At these instances (corresponding to the mining of a block), a jump down occurs, where the jump size is a random variable $B$ that corresponds to the size of a block that is mined. For $t\geq 0$ and the i.i.d. sequence $(B_i)_{i\geq 1}$ that has the same distribution as $B$, we consider the stochastic process:
\begin{equation}\label{eq:def-Y(t)}
   Y_{y_\phi,c_\phi,\lambda,B}(t):=  y_\phi + c_\phi t -\sum_{i=1}^{Z_\phi(t)} B_i.
\end{equation}
Block sizes are taken fixed (deterministic) and scaled to size 1 (in Section \ref{sec:queueing_model} we support this modelling choice). Thus, $F_B(x)=\1\{x\geq 1\}$, where $\1\{\cdot\}$ denotes an indicator function.
We furthermore scale time such that the arrival rate of mined blocks equals 1, i.e.  $\lambda=1$. 
This  is a special case of the so-called Cram\'er-Lundberg (CL) model. The CL model is fundamental to actuarial sciences, see  \cite[Chapter IV]{asmussen2020} for more details. 

In what follows, we assume $\phi$ to be fixed, and therefore suppress its presence in the notation for convenience.
Additionally, when it is absolutely clear from the context, we also suppress the subscripts and simply write  $Y(t)$ instead of $Y_{y_\phi,c_\phi,\lambda,B}(t)$.

As a key consequence,
\textit{the confirmation time} of a Bitcoin transaction is modelled  by the time required for the stochastic process to hit 0 for the first  time, (in the actuarial  science jargon, this  is the  so-called time to ruin, $T_{y,c,\lambda,B}$), i.e.
\begin{equation}\label{eq:def-confirmation}
    T_{y,c,\lambda,B} = \inf\{t>0: Y_{y,c,\lambda,B}(t) < 0 \}.
\end{equation}
For similar reasons as above, we suppress the subscripts $c$, $\lambda$ and $B$ in $T_{y,c,\lambda,B}$ when these values are clear from the context.

We derive the confirmation time distribution in a generic mathematical framework. In this analysis we take into account that the term \lq confirmation time\rq, may be ambiguously used in bitcoin literature and that the analysis can be adjusted to different definitions of \lq confirmation time\rq, that are often used. Remark \ref{rem:cnofirmationtime_def} describes this generalisation in more detail.

\begin{remark}[The definition of a confirmation time]
\label{rem:cnofirmationtime_def}
Users sometimes consider a transaction confirmed only when a certain number, say $z$, of blocks are found after the block in which the transaction of interest is included. This has to do with a possible \textit{double-spend attack}, see \cite{Pinzon2016}. 
Note that we have opted for simplicity to define the  confirmation time with $z=0$. Under the current modelling assumptions the case for $z>0$ is given by
$T_y+E_1+\ldots+E_z$, where $T_y$ is given in 
Equation \eqref{eq:def-confirmation} and $\{E_i\}_{i=1}^z$ are i.i.d. exponentially distributed times with parameter 1. All results presented in the paper can be extended in a straightforward manner to cover the case $z>0$.
\end{remark}

The density function of the confirmation time is derived in the following proposition: 
\begin{proposition}[Density of  confirmation time]
\label{col:Conftimedens}
Consider the process $\{Y(t)\}_{t\geq0}$ as defined in Equation \eqref{eq:def-Y(t)} with initial level $Y(0)=y$. 
Let $T_y$ be the  confirmation time defined in Equation  \eqref{eq:def-confirmation}. 
Then its  density function $f_{T_y}(t)$, $y>0$, is given by
\begin{equation}
    \label{eq:densityut}
    \begin{aligned}
    f_{T_y}&(t) =
    \1\{y+ct <1\}\e^{-t}\\& +
    \1\{y+ct \geq 1\}\left[ \e^{-t} \frac{ t^{\lfloor y+ct\rfloor}}{\lfloor y+ct\rfloor!}-
    \sum_{i=\lceil y \rceil}^{\lfloor y+ct\rfloor}
   f_{T_0}\left(t-\frac{i-y}{c}\right) \e^{-\frac{i-y}{c}} \frac{ \left(\frac{i-y}{c}\right)^i}{i!}\right],
    \end{aligned}
\end{equation}
where
\begin{equation}
	\label{eq:densityt0}
    f_{T_0}(t) = \1\{ct < 1\} \e^{-t}+
	\1\{ct \geq 1\} \frac{ct-\lfloor ct\rfloor}{ct}\e^{-t}
	\frac{ t^{\lfloor ct \rfloor}}{\lfloor ct\rfloor! }.
	\end{equation}
	\label{cor:T(y,c)density}
\end{proposition}
The proof of the proposition is presented in Section \ref{sec:cor1}.
We note that Proposition \ref{cor:T(y,c)density} measures $T_y$ in continuous time, but users are often also interested in the number of blocks until confirmation. This form of the confirmation time can be viewed as the  discrete time analogue of $T_y$: Let $N_y$ be the number of blocks to confirmation, then 
\begin{align}
    N_y &:= Z(T_y)\nonumber\\
        &:=\min_t\{Z(t): Y(t)<0\} \label{eq:def_Ny}.
\end{align}
We next provide the tail (or complementary) distribution of $N_y$ in the following proposition:

\begin{proposition}[Confirmation time in blocks]
\label{prop:DM1BP}
Consider the process $\{Y(t)\}_{t\geq0}$ as defined in Equation \eqref{eq:def-Y(t)} with initial level $Y(0)=y=m+\varepsilon$, with $m\in\mathbb{N}_0$ and $\varepsilon\in[0,1)$. Then the tail distribution of the number of blocks to confirmation $N_y$, as defined in Equation \eqref{eq:def_Ny},
is given by: $\prob[ N_y > m ] =1$ and, for $n\in \mathbb{N}$,
\begin{equation}
    \prob[ N_y>n+m ]
    =
    \sum_{\vec k_n\in K_n(m)}
    \e^{-\frac{1}{c}(n-\varepsilon)}
     \frac{ (1/c)^{k_1+\ldots +k_{n-1}} }{k_0!\cdots k_{n-1}!}
     ((1/c)(1-\varepsilon))^{k_0},
\label{eq:stadjey=y}
\end{equation}
with $\vec{k}_n=(k_0,\ldots,k_{n-1})$ and $K_n(m)=\{ \vec{k}_n \in \mathbb{N}_0^n \mid k_0+\cdots +k_{i}\leq i+m, \mbox{ for all } i=0,1,\ldots,n-1\}$.
\end{proposition}
The proof of Proposition \ref{prop:DM1BP} is presented in Section \ref{sec:prop2}.
While the density function in Proposition \ref{cor:T(y,c)density} and the tail distribution function in Proposition \ref{prop:DM1BP} are explicit,  using it to compute the expectation is not straightforward and may be time-consuming in practice. Therefore, we  derive directly an explicit expression for  both the mean time and number of blocks to confirmation:

\begin{proposition}[The expected confirmation time]
\label{prop:ET}
Consider the process $\{Y(t)\}_{t\geq0}$ as defined in Equation \eqref{eq:def-Y(t)} with initial level $Y(0)=y=m+\varepsilon$, with $m\in\mathbb{N}_0$ and $\varepsilon\in[0,1)$. Let $T_y$ and  $N_y$ be the  confirmation times defined in Equation  \eqref{eq:def-confirmation} and Equation \eqref{eq:def_Ny} respectively. Then 
$\expec[N_y]=\expec[T_y]$ and
for $y\geq 1$, the expected confirmation time is recursively given by
    \begin{equation}
    \label{eq:E(T)alg}
   \mathbb{E}[T_y] =
   \int_0^{(1-\varepsilon)/c}
   (x+\mathbb{E}[T_{y+cx-1}]) \e^{-x}dx +
   \e^{-\frac{1-\varepsilon}{c}} \left(\tfrac{1-\varepsilon}{c}+ \mathbb{E}[T_{m+1}]\right)
    \end{equation}
    and in case $y=\varepsilon<1$,
\begin{equation}\label{eq:E(t)y<1}
\mathbb{E}[T_y] = \int_0^{(1-\varepsilon)/c}
   x \e^{-x}dx +
   \e^{-\frac{1-\varepsilon}{c}} \left[ \frac{1-\varepsilon}{c}+
   \frac{1}{c\rho}e^{1/c} + 1 - e^{1/c}\right],
\end{equation}
    with $\mathbb{E}[T_0]=0$. 
    Here, $\rho$ is the (unique) positive solution of the equation
    $cy - 1 + \e^{-y} = 0$. 
\end{proposition}
While Proposition \ref{prop:ET} determines $\expec[T_y]$ implicitly, we provide a detailed iterative scheme on how one can determine $\expec[T_y]$ for all $y\geq 0$ in Section \ref{sec:prop3}.

\section{Limiting results for the batch-service queue (BSQ)}
\label{sec:queueing_model}
As detailed in Section \ref{sec:relation to bc lit}, the BSQ has been used to model the Bitcoin mempool before. However, such models had their shortcomings, in both their specific modelling assumptions and limited tractable results. 
To mitigate the former, we consider a more realistic setting where dependencies between the transaction arrival process and block sizes is taken into account. To mitigate the latter, we take an appropriate limit that scales the BSQ to either its fluid or diffusion limit. In Section \ref{sec:case_study} we reflect on the practical implications of the scaling and how they can be exploited in real-world applications.

We first show that the BSQ has a fluid limit for the transaction count process and the cumulative weight process, which, for both, is a CL model.
We then show that if we additionally scale the rate of incoming transactions to the rate of confirmed transactions, the BSQ has a diffusion limit for the transaction count process and the cumulative weight process, which is, for both, a Brownian motion.
Finally, we show that this also implies that the confirmation time of the BSQ (also called hitting time, i.e. the first time a process hits 0) converges to the confirmation time of the CL model and the Brownian motion in the fluid and diffusion limit setting respectively.

We first provide details on the BSQ.
We assume that 
transactions arrive according to a Poisson process $A_\nu(t)$ with parameter $\nu$, where the $i-th$ transaction is of weight $X_i$. We assume $\{X_i\}_{i\geq 1}$ are i.i.d. copies of $X$.
 Next, we assume a block mining (that represents a batch-service) that has size $K$ occurs according to a Poisson process $D_\lambda(t)$ with parameter $\lambda$.
 To avoid trivialities, we assume $X\ll K$ w.p. 1.
 This process is independent of the number of transactions in the queue.
 During a block mining, as many transactions as possible are confirmed in a first-come-first-serve manner, such that the cumulative weight of these transactions does not exceed $K$. In case the cumulative weight of all transactions is below $K$, the entire queue is confirmed at the finding of a block. Finally, we assume that at time $t=0$ there are $m$ transactions in the system, of which the weights are observable.
 It is noteworthy that these transactions may not be i.i.d., as one transaction was too large in size to fit into the last confirmed block. However, this is taken into account for the first $m$ transactions in Equation \eqref{eq:fluid_initial_conditions} and in Section \ref{app:fluid proof workload} for future transactions.
 
 Denote the cumulative weight process, (coinciding with the workload in the BSQ) and the transactions count process (coinciding with the total number of customers in the BSQ), at time $t$ by $\Q_{m,\nu,\lambda,K}^*(t)$ and $\M_{m,\nu,\lambda,K}^*(t)$ respectively.
For convenience, we suppress the subscripts in the subsequent analysis.
 Furthermore, assume $A_\nu(t)$ and $D_\lambda(t)$ are independent and also independent of all transaction weights $\{X_i\}_{i\geq1}$. 
 Next,
let $T_{\Q^*}$ and $T_{\M^*}$  be the confirmation times of $\Q^*(t)$ and $\M^*(t)$ , i.e. the first time these processes cross 0, then 
\begin{equation}
T_{\Q^*} = \inf_{t>0}\{ \Q^*(t) = 0 \}=
\inf_{t>0}\{ \M^*(t) = 0 \}=
T_{\M^*}
.
\end{equation}
These first confirmation times can be understood as the busy period time of the BSQ with some initial work.
Directly after time $T_{\Q^*}$, we observe that in practice $\Q^*(T_{\Q^*}^+)$ is set to $0$ after which it continues as described above, and a similar process occurs for $\M^*(T_{\M^*}^+)$. For fixed $t$, it is generally 
complicated to find elegant expressions for $\Q^*(t)$ and $\M^*(t)$. However, when restricting 
$t\in[0,T_{\Q^*}]$ we can define two simpler auxiliary processes as follows. First, define
\begin{equation}
\label{eq:defQ}
      \Q(t) = \sum_{h=1}^m X_h + \sum^{m+A_\nu(t)}_{i=m+1} X_i -\sum_{j=1}^{D_\lambda(t)} \big(K-U_j(K)\big) ,
\end{equation}
where there are $m$ observable initial transactions in the system and where $U_j(K)$ denotes the space left over in the $j$-th block mining of size $K$ that could not be filled with the next transaction. We refer to Equation \eqref{eq:defUk} for a formal definition of the latter expression. Second, define
\begin{equation}
\label{eq:defM}
    \M(t) = m+ A_\nu(t)-\sum_{j=1}^{D_\lambda(t)} B_j(K),
\end{equation}
where $m$ denotes the observable initial number of transactions and $B_j(K)$ denotes the number of transactions that can fit in the $j$-th batch of size $K$, we refer to Equation \eqref{eq:defB} for a formal definition. Note that this imposes a natural dependence of the different terms in $\Q(t)$ and $\M(t)$.

Next, it is straightforward to verify that
$T_{\Q^*} \overset{d}{=} T_{\Q}$ and $T_{\M^*} \overset{d}{=} T_{\M}$. Moreover,
\begin{equation}
\label{eq:Q=Q*}
    (\Q(t))_{t\in[0,T_{\Q}]}=  (\Q^*(t))_{t\in[0,T_{\Q^*}]},
\end{equation}
and
\begin{equation}
\label{eq:M=M*}
     (\M(t))_{t\in[0,T_{\M}]}=  (\M^*(t))_{t\in[0,T_{\M^*}]}.
\end{equation}

Recall $\Y(t)$ is a CL model as defined in Equation \eqref{eq:def-Y(t)} with starting position $y$, slope $c$, block arrival rate $\lambda$ and with block size 1.
We recall from Equation \eqref{eq:def-confirmation}, $T_{y,c,\lambda,1}$ as the confirmation time of $\Y(t)$, i.e. $T_{y,c,\lambda,1} = \inf_{t>0}\{\Y(t)\leq 0\}$.
Our main contribution then regards the convergence of the confirmation times of the appropriately scaled BSQ, both for the cumulative weights process $\Q(t)$ and the transaction count process $\M(t)$, to the confirmation time of the CL model.
Such a limit is also known as a \textit{fluid limit}, see \cite[Chapter 6]{Chen2001},  for more details.
\begin{theorem}[Convergence of the confirmation time via a fluid limit]
\label{thm:fluidlimit}
Consider a BSQ as described above. Assume that the i.i.d. weights of transactions have the same distribution as $X$ and assume $\nu\expec[X]<K\lambda$.
For some given integer function $m(n)$, define 
\begin{equation}
\label{eq:defQn}
    \Q^{(n)}(t) =\sum_{h=1}^{m(n)}X_h+ \sum^{m(n)+A_{\nu n}(t)}_{i=m(n)+1} X_i -\sum_{j=1}^{D_\lambda(t)} (nK-U_j(nK)),  \end{equation}
and
\begin{equation}
\label{eq:defMn}
      \M^{(n)}(t) =m(n) +A_{\nu n}(t)-\sum_{j=1}^{D_\lambda(t)} B_j(Kn).
\end{equation}   
Assume furthermore that, for fixed $y\geq 0$, $m(n)$ is such that
\begin{equation}
\label{eq:fluid_initial_conditions}
    \lim_{n\to\infty}
    \frac{1}{n}
    \Q^{(n)}(0) = Ky,
    \qquad 
    \text{ and }
    \qquad
    \lim_{n\to\infty}
    \frac{1}{n}
    \M^{(n)}(0) = \frac{Ky}{\expec[X]}.
\end{equation}
Recall $T_{\Q^{(n)}}$ and $T_{\M^{(n)}}$ as the confirmation times of the above processes respectively.    
    Then 
\begin{equation}
\label{eq:thm main result}
    T_{\frac{\Q^{(n)}}{Kn}} \xrightarrow{d}
    T_{y,\nu\expec[X]/K,\lambda}\,,
    \quad 
    \text{ and }
    \quad 
    T_{\frac{\expec[X] \M^{(n)}}{Kn}} \xrightarrow{d}
    T_{y,\nu\expec[X]/K,\lambda}.
\end{equation}
\end{theorem}
Intuitively, Theorem \ref{thm:fluidlimit} states that $\Q^{(n)}(t)/(Kn)$ and $\expec[X]\M^{(n)}(t)/(Kn)$, i.e. the scaled BSQ where the rate of incoming transactions and blocksizes are proportionally scaled to infinity,  resembles the CL model.
Such a scaling is realistic in practice as discussed in Section \ref{sec:relation to bc lit}.
Moreover, the result states that the confirmation time converges accordingly. We refer to Figure \ref{fig:fluidlimits} for a visual representation.
\begin{figure}[ht]
    \centering
    \includegraphics[scale=0.7]{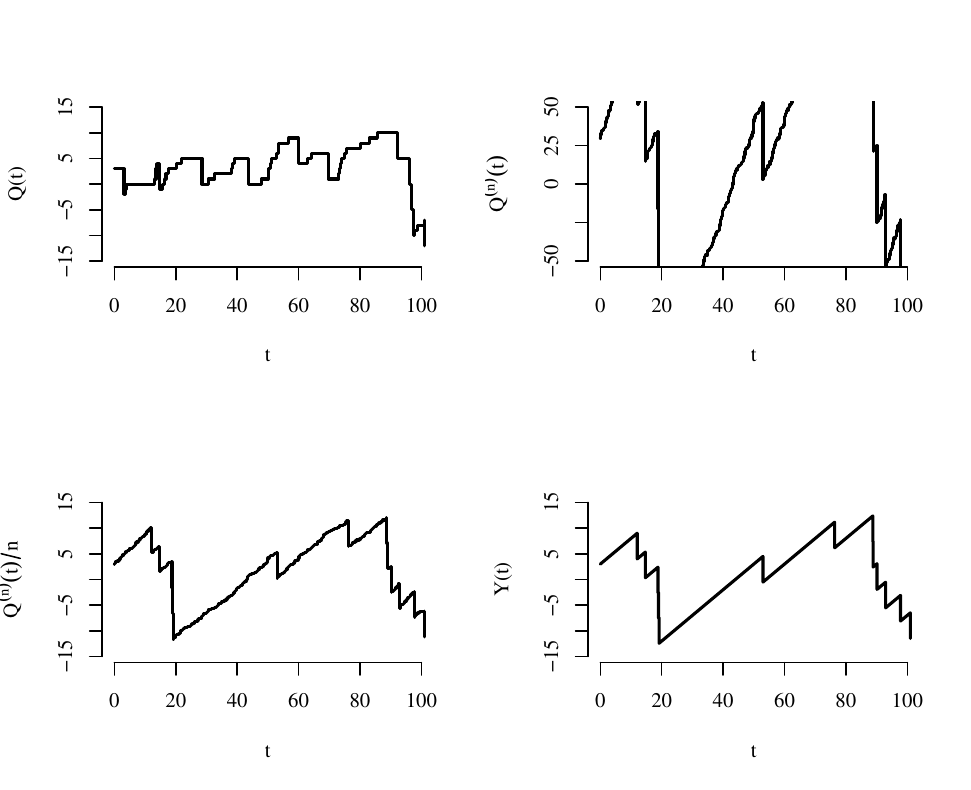}
    \caption{Example of the different stochastic processes introduced above.}
    \label{fig:fluidlimits}
\end{figure}
Based on our observation in Equation \eqref{eq:Q=Q*} and Equation \eqref{eq:M=M*}, Theorem \ref{thm:fluidlimit} is sufficient to show that the confirmation times of $Q_{m,\nu,\lambda,K}^{*}$ and $M_{m,\nu,\lambda,K}^{*}$ satisfy the same limit. However, the proof for $Q^{(n)}(t)$ and $M^{(n)}(t)$ is significantly simpler and is detailed in Section \ref{sec:fluidlimit proof}.

Finally, we provide our second result for the confirmation time of a BSQ scaled to a \textit{heavy-traffic regime}, i.e. where, additionally to the scaling in Theorem \ref{thm:fluidlimit}, the rate of incoming transactions converges to the rate of outgoing transactions. Such a limit is also known as a \textit{diffusion limit}, see for example \cite{Whitt2002}. We show that in this regime the transactions count process and the cumulative weight process of the BSQ have a diffusion limit, which is a Brownian motion (BM). Such a diffusion limit was argued before in \cite{Gundlach2021} but then for the CL model. One can, in a sense, already combine that result with Theorem \ref{thm:fluidlimit} to show a heavy-traffic limit of a BSQ to a BM, but then via the CL model.
However, the following result shows a direct convergence of the BSQ to the BM
without needing the CL as an intermediate step.

\begin{theorem}[Convergence of the confirmation time via a diffusion limit] 
\label{thm:diffusion_limit}
Consider $\Q^{(n)}(t)$ and $\M^{(n)}(t)$ from \eqref{eq:defQn} and \eqref{eq:defMn} respectively
and assume $\nu\expec[X]<K\lambda$. Take $\sigma=\sqrt{t\lambda}$ and set $n=\sigma^2K^2/(\lambda K-\nu\expec[X])$ and assume that $m(n)$ is such that
\begin{equation}
\label{eq:diffusion_initial_cond}
    \lim_{n\to\infty} 
    \frac{1}{n\sqrt{n}}
    \Q^{(n)}(0) = \sigma K y,
    \qquad 
    \text{ and }
    \qquad
    \lim_{n\to\infty} 
    \frac{1}{n\sqrt{n}}
    \M^{(n)}(0) = \sigma K y.
\end{equation}
    Furthermore, let $T^*_{y,c}$ be the confirmation time of a Brownian motion
    $B_{y,c}(t)$
    starting in $y>0$ with drift $c<0$ and with variance $1$. Then for $n\to\infty$, we find
    \begin{equation}
        T_{
        \frac{Q^{(n)(nt)}}{
        (K\sigma n\sqrt{n})} }
        \xrightarrow{d} 
        T^*_{y,-1 },
        \qquad 
        \text{ and }
        \qquad
        T_{
        \frac{\expec[X]M^{(n)}(nt)}{
        (K\sigma n\sqrt{n})} 
        }\xrightarrow{d} 
        T^*_{y,-1 }.
    \end{equation}
\end{theorem}
The result follows from the diffusion limit that shows convergence of the BSQ to the BM, see Sections \ref{app:dif proof workload} and \ref{app:dif proof cust count} for more details and the proof. 
It is well known that the time for a BM to hit a certain level follows an inverse Gaussian distribution, see \cite{Folks1978}. 
Moreover, one can rescale time to find an approximation for $T_{\Q^{(n)}/(nK)}$ for $n$ large and $\lambda K \approx \nu \expec[X]$.
However, if one does so, an appropriate correction to account for the undershoot of the CL should be implemented, c.f. \cite[Chapter V6]{asmussen2020}.
For details on the confirmation time of the BM and the correction we refer to \cite{Gundlach2021} and Remark \ref{rem:diffusionapprox}.

\begin{remark}[Correction for the undershoot]
\label{rem:diffusionapprox}

Under the assumptions of Theorem \ref{thm:diffusion_limit} one can scale time and include a correction for the expected undershoot to find 
\begin{equation}
 \label{eq:conftimedistradj}
    T_{\Q^{(n)}/(Kn)} \approx
 \text{IG}\left(\frac{y+\mathbb{E}[S_{y}]}{1-c}, (y+\mathbb{E}[S_{y}])^2\right).
\end{equation}
We find therefore a practical estimation for the expected confirmation time by
\begin{equation}
\label{eq:E(t)approx}
    \mathbb{E}[T_y] \approx
     \frac{ y + \mathbb{E}[S_{y}]}{1-c}.
\end{equation}
Here $\mathbb{E}[S_{y}]$ denotes the expected undershoot below level 0 of a CL model with parameters $y$ and $c$. The expected undershoot $\mathbb{E}[S_{y}]$ satisfies, for all $y\geq 1$ non-integer-valued,
\begin{align}
\mathbb{E}[S_{y}] &=\int_{0}^{\frac{\lceil y \rceil-y}{c}} \mathbb{E}[S_{y+cx-1}]e^{-x}\mathrm{d}x+e^{-\frac{\lceil y\rceil -y}{c}} \mathbb{E}[S_{\lceil y \rceil}],
\end{align}
for  $y\geq 1$ integer-valued ($y\in \mathbb{N}$),
\begin{align}
\mathbb{E}[S_{y}]&=\int_{0}^{\frac{1}{c}} \mathbb{E}[S_{y+cx-1,c}]e^{-x}\mathrm{d}x+e^{-\frac{1}{c}} \mathbb{E}[S_{y+1}]
\end{align}
and, for $y\in[0,1]$,
\begin{align}
\mathbb{E}[S_{y}]&=1-y-c+e^{-\frac{1-y}{c}}\left(c+\mathbb{E}[S_{1}]\right).
\end{align}

We refer the interested reader to \cite{Gundlach2021} for more details on the approximation approach and we refer to Section \ref{sec:results}, for an indication on the accuracy of this approximation. Note that even for $c$ further away from 1, this approximation performs relatively well.
\end{remark}

\section{Confirmation times in the CL model}
\label{sec:math_analysis}
\vspace{0.5cm}
\subsection{Proof of Proposition \ref{cor:T(y,c)density} -- Density of the confirmation time}
\label{sec:cor1}
A more general version of Proposition  \ref{col:Conftimedens} for the probability density function of the time to ruin, with  a proof, is given in \cite[Section 8.6]{Dickson2016}. For the situation at  hand, due primarily  to the deterministic block sizes, $F_B(x)=\1\{x\geq 1\}$, the resulting expression is more compact and the proof is significantly shorter and easier. Moreover, the intuition behind each term in Equation \eqref{eq:densityut} and Equation \eqref{eq:densityt0} is provided.

\begin{proof}[Proof of Proposition  \ref{col:Conftimedens}.]

Let us start with a proof of Equation \eqref{eq:densityut}.
We distinguish between the case $y+ct<1$ and the case $y+ct \geq 1$. For the former case, a confirmation within $[t,t+dt]$ is attained if and only if the first block arrives in $[t,t+dt]$, which happens with probability $\e^{-t} dt$. Next, we focus on the case $y+ct \geq 1$. Note that, due to the fact that block sizes are deterministic, in order to have a confirmation in $[t,t+dt]$, exactly $\lfloor y+ct\rfloor$ blocks need to arrive before time $t$, in order to have $Y(t)\in(0,1)$.
Recall that $Z(t)$ is a standard Poisson process with mean inter-arrival time equal to 1, then
\begin{equation*}
\begin{aligned}
&\mathbb{P} [T_y \in [t,t+dt]\ ] \\
&=
\mathbb{P} [Z(t) = \lfloor y+ct\rfloor,  Y(u) \geq 0 \text{ for all $u<t$}, \text{block arrival in } [t,t+dt]\ ]\\
&= \mathbb{P} [Z(t) = \lfloor y+ct\rfloor,\text{ block arrival in } [t,t+dt]\ ]\\
&\quad -
\mathbb{P} [Z(t) = \lfloor y+ct\rfloor, \text{ process $Y(t)$ crossed the $x$-axis before time $t$}, \text{block arrival in } [t,t+dt]\ ].
\end{aligned}
\end{equation*}
Due to the memoryless property of the Poisson process, the expression for the first part is given by
\begin{align}
\mathbb{P} [Z(t) &= \lfloor y+ct\rfloor,\text{ block arrival in } [t,t+dt]\ ] \nonumber\\
&=
\e^{-t} \frac{ t^{\lfloor y+ct\rfloor}}{\lfloor y+ct\rfloor!}
\mathbb{P} [\text{block arrival in } [0,dt]\ ] = \e^{-t} \frac{ t^{\lfloor y+ct\rfloor}}{\lfloor y+ct\rfloor!} dt.
\label{eq:densityutanycross}
\end{align}
It remains to find
\begin{equation*}
\mathbb{P} [Z(t) = \lfloor y+ct\rfloor, \text{ process $Y(t)$ crossed the $x$-axis before time $t$}, \text{block arrival in } [t,t+dt]\ ].
\end{equation*}
Distinguish between a \textit{down-crossing} at time $v$, denoting an instant at which the process $Y(t)$ crosses the $x$-axis from above, and an \textit{up-crossing} at time $w$, denoting a subsequent instant at which the process $Y(t)$ crosses the $x$-axis from below again $(0<v<w<t)$. Note that multiple up- and down-crossings can occur.
A key notion to the proof is that, while down-crossings can occur at any time, up-crossings can only occur on specific times due to the constant increase of the process $Y(t)$ and the fact that the sizes of downward jumps are always equal to $1$.
Denote these possible up-crossing times by the set $\mathcal{V}$, given by
\begin{equation}
\mathcal{V} =
\left\lbrace
\frac{ \lceil y\rceil -y}{c}
, \frac{\lceil y\rceil -y+1}{c}
,\ldots ,
\frac{\lfloor y+ct\rfloor  - y}{c}
\right\rbrace .
\end{equation}
Condition on the last up-crossing and note that from there the process $Y(t)$ starts, independent of the past, again as a process with initial level 0. We can therefore write
\begin{align}
\mathbb{P} [Z(t) =& \lfloor y+ct\rfloor, \text{ process $Y(t)$ crossed the $x$-axis before time $t$}, \text{block arrival at } [t,t+dt]\ ] \nonumber\\
&=
\sum_{v\in\mathcal{V}}
\mathbb{P} [Z(t) = \lfloor y+ct\rfloor,\text{ Last up-crossing at } v, \text{ block arrival at } [t,t+dt]\ ]\nonumber\\
&=
\sum_{v\in \mathcal{V}}
\mathbb{P} [\text{Up-crossing at } v, \ T_0 \in [t-v,t+dt-v]]\nonumber\\
&=
\sum_{v\in \mathcal{V}}
f_{T_0}(t-v) \e^{-v} \frac{v^{cv+y}}{(cv+y)!}dt \label{eq:densityutupcros}
\end{align}
where the last step follows from the fact that in order to have an up-crossing at $v$, exactly $cv+y$ blocks need to have arrived before time $v$.
Combining the results of Equation \eqref{eq:densityutanycross} and Equation \eqref{eq:densityutupcros}, determines the density (after scaling with $dt$ and taking the limit of $dt\to0$) in the case where $y+ct \geq 1$. This completes the proof of  Equation \eqref{eq:densityut}.

It remains to prove Equation \eqref{eq:densityt0}. Again, the case $ct<1$ follows directly, and we therefore focus on the case $ct \geq 1$. We use again the property that before time $t$, exactly $\lfloor ct\rfloor$ blocks need to have arrived, in order to have that $Y(t^-) \in (0,1)$ and a confirmation can occur at time $[t,t+dt]$. Hence,
\begin{align*}
	&\mathbb{P} [ T_0 \in [t,t+dt] ] \\
	& =
	\mathbb{P} [ Z(t)=\lfloor ct\rfloor , \text{ process $Y(t)$ does not cross $x$ axis before $t$}, \text{block arrival at } [t,t+dt]\ ] \\
	&=
\mathbb{P} [\text{ process $Y(t)$ does not cross $x$ axis before $t$} ]\mid Z(t)=\lfloor ct\rfloor\ ] \e^{-t}
	\frac{ t^{\lfloor ct \rfloor}}{\lfloor ct\rfloor! } dt,
	\end{align*}
	and it is sufficient to show
	\begin{equation}
	\mathbb{P} [\text{ Process $Y(t)$ does not cross $x$ axis before $t$}\mid  Z(t)=\lfloor ct\rfloor\ ] = \frac{ ct - \lfloor ct\rfloor }{ct}.
	\end{equation}
The event that the process $Y(t)$ does not cross the $x$-axis before time $t$ is equivalent to the event that the Poisson process $Z(t)$ is not crossing the line $y=ct$ before time $t$. Therefore, we are interested in
		\begin{equation*}
		\mathbb{P} [Z(v)< cv,\ v\in(0,t)\mid Z(t) = \lfloor ct \rfloor ].
		\end{equation*}
		In case $c=1$ one could directly apply the continuous-time ballot theorem, see for example \cite{Konstantopoulos1995}. 
		Otherwise, define $Z^*(t)=Z(t)/c$, then
		\begin{equation*}
		\mathbb{P} [Z(v)< cv,\ v\in(0,t)\mid Z(t) = \lfloor ct \rfloor )=
		\mathbb{P} \left[ Z^*(v)<v,\ v\in(0,t)\mid Z^*(t)=\frac{\lfloor ct\rfloor}{c}\right]=1-\frac{\lfloor ct\rfloor}{ct},
		\end{equation*}
		from which the claim follows.
\end{proof}

\subsection{Proof of Proposition \ref{prop:DM1BP} -- Distribution of the number of blocks to confirmation}
\label{sec:prop2}
We continue with the proof of Proposition \ref{prop:DM1BP} and derive the tail distribution of the number of blocks to confirmation.

\begin{proof}[Proof of Proposition \ref{prop:DM1BP}]
Let $k_0$ be the number of block arrivals in the time interval $(0,\tfrac{1-\varepsilon}{c})$ and $k_i$ be the number of block arrivals in the time interval $(\tfrac{i-\varepsilon}{c}, \tfrac{i+1-\varepsilon}{c})$ for $i \geq 1$. Then in order to have that $N_y > n+m$ we should have that $\vec k_n\in K_n(m)$. Furthermore, for a Poisson process with rate $1$ we have that the probability that $k_0$ arrivals occur in the interval $(0,\tfrac{1-\varepsilon}{c})$ and $k_i$ arrivals occur in the interval $(\tfrac{i-\varepsilon}{c}, \tfrac{i+1-\varepsilon}{c})$ for $i=1,\ldots,n-1$ is equal to
\begin{equation}
\e^{-\frac{1}{c}(1-\varepsilon)}\frac{((1/c)(1-\varepsilon))^{k_0}}{k_0!} \prod_{i=1}^{n-1} \left( \e^{-\frac{1}{c}}\frac{(1/c)^{k_i}}{k_i!} \right) = \e^{-\frac{1}{c}(n-\varepsilon)}
     \frac{ (1/c)^{k_1+\ldots +k_{n-1}} }{k_0!\cdots k_{n-1}!}
     ((1/c)(1-\varepsilon))^{k_0}.
\end{equation}
Hence, the proposition follows.
\end{proof}

We note, however, that similar results can be found in the field of queueing theory. To do so, one needs to formalise the duality between the CL model and a queueing model, which we do in Appendix \ref{app:Alternative proof Exp conf time}. Based on this duality, an alternative proof is derived as well.

\subsection{Proof of Proposition \ref{prop:ET} -- Mean confirmation time}
\label{sec:prop3}
Next, we consider the mean confirmation time. First, we argue that, due to the scaling of time and space, the mean confirmation time in continuous or discrete time is the same:

\begin{theorem}[Expected time to confirmation]\label{lem:T=N}
Consider the process $\{Y(t)\}_{t\geq0}$ as defined in Equation \eqref{eq:def-Y(t)}. Let $T_y$ be the  confirmation time defined in Equation  \eqref{eq:def-confirmation} and  let $N_y := \inf\left\{k\geq 1: y+c\sum_{i=1}^k A_i< k\right\}$ denote  the number of blocks to confirmation. Then
\begin{equation}
\label{eq:T=N}
    \mathbb{E}[N_y]=\mathbb{E}[T_y].
\end{equation}
\end{theorem}
\begin{proof}
The relation between $N_y$ and $T_y$ is given by
   \begin{equation}
       T_y = \sum^{N_y}_{i=1}A_i.
   \end{equation}
Note that $N_y$ is a stopping time with respect to $A_i$, as confirmation occurs directly after a block has been mined. Thus, the event $\{N_y=n\}$ is independent of $\{A_{n+i}\}_{i\geq 1}$. Furthermore, as $c<1$, $\mathbb{E}[N_y]<\infty$ and $\mathbb{E}[A_i]=1$, all properties of Wald's identity (\cite{Wald1945}) are satisfied yielding $\mathbb{E}[T_y]=\mathbb{E}[A_i]\mathbb{E}[N_y]$, which directly proves the claim.
\end{proof}
The time to ruin of the CL model has been studied before in the literature. We adapt this result for starting position $y\in\mathbb{N}$ as follows:

\begin{theorem}[Expected time to confirmation for integer valued $y$.]
\label{prop:ET_N}
Consider the process $\{Y(t)\}_{t\geq0}$ as defined in Equation \eqref{eq:def-Y(t)}. Let $T_y$ be the  confirmation time defined in Equation  \eqref{eq:def-confirmation}. Then, for $y\in \mathbb{N}$, the expected time to confirmation is given by
    \begin{equation}
    \label{eq:E(T)}
    \mathbb{E}[T_y]=
    y+1 + \sum_{i=0}^y
    \frac{ \frac{e^{\frac{i}{c}}}{\rho c} \left(\frac{-i}{c}\right)^{y-i} - \Gamma(1+y-i,-i/c)}{(y-i)!},
    \end{equation}
    where $\rho$ is the (unique) positive solution of the equation
    $c \rho - 1 + \e^{-\rho} = 0$ and $\Gamma(s,x) = \int_{t=x}^{\infty} t^{s-1} e^{-t} dt$ is the upper incomplete gamma function.

\end{theorem}
\begin{proof}
   In \cite{Picard1998}, the first moment is computed
\begin{equation}
    \mathbb{E}[T_y]=
    \frac{1}{c\rho}L(y) + \sum_{r=0}^y (1 - L(y-r))
\end{equation}
with
\begin{equation}
    L(y) =
    \sum_{i=0}^y e^{\frac{ i}{c}}
    \frac{\left(\frac{-i}{c}\right)^{y - i}}{(y - i)!}.
\end{equation}
Substituting $L(y)$ and interchanging the sums in the right-hand side expression, simplifies the expression to Equation \eqref{eq:E(T)}. Here we used the fact that, for a positive integer $s$, the upper incomplete Gamma function is equal to
\begin{equation}
\Gamma(s,x) = (s-1)! e^{-x} \sum_{k=0}^{s-1} \tfrac{x^k}{k!}.
\end{equation}
\end{proof}

Next, we extend this result in order to find the mean confirmation time for any $y\in\mathbb{R}^+$
and thereby proving Proposition \ref{prop:ET}.

\begin{proof}[Proof of Proposition \ref{prop:ET}]
Suppose that $y=m+\varepsilon$, where $m\in\mathbb{N}$ and $\varepsilon\in(0,1)$, then by conditioning on the first block arrival instant,
\begin{equation}\label{eq:E(T)alg_proof}
\mathbb{E}[T_y] =
   \int_0^{(1-\varepsilon)/c}
   (x+\mathbb{E}[T_{y+cx-1}]) \e^{-x}dx +
   \e^{-\frac{1-\varepsilon}{c}} \left(\tfrac{1-\varepsilon}{c}+ \mathbb{E}[T_{m+1}]\right).
\end{equation}
This proves Equation \eqref{eq:E(T)alg}.
In case $y=\varepsilon<1$, 
we used that $\mathbb{E}[T_{y+cx-1}] =0$ for $x < (1-y)/c$ and that $\mathbb{E}[T_1] = \tfrac{1}{c\rho} e^{1/c} + 1 - e^{1/c}$ which follows from substituting $y=1$ in Equation \eqref{eq:E(T)}.
Equation \eqref{eq:E(T)} reduces then to Equation \eqref{eq:E(t)y<1} which concludes the result.
\end{proof}

Equations \eqref{eq:E(T)alg}, \eqref{eq:E(t)y<1} and
\eqref{eq:E(T)} can be used to determine $\mathbb{E}[T_y]$ in an iterative manner. The case $y\in[0,1]$ is covered by Equation \eqref{eq:E(t)y<1}. With this, the case $y=2$ follows from Equation \eqref{eq:E(T)} and the case
$1<y<2$ follows from Equation \eqref{eq:E(T)alg}. The case $y=3$ and $2<y<3$ follows again from Equation \eqref{eq:E(T)} and Equation \eqref{eq:E(T)alg} respectively, and so on.
This directly implies that we can obtain the expected confirmation time for an arbitrary positive value of $y$. 

\section{Fluid and diffusion limits for the BSQ}
\label{sec:fluidlimit proof}
In the following we provide more details on the claims made in Section \ref{sec:queueing_model} and prove Theorems \ref{thm:fluidlimit} and \ref{thm:diffusion_limit}.
As  $Q^{(n)}(t)$ and $M^{(n)}(t)$ in Equation \eqref{eq:defQn} and Equation \eqref{eq:defMn} respectively are c\`adl\`ag (right-continuous and have a limit from the left), convergence of these  processes can be shown in the Skorohod $D[0,\tau]$ space, for some $\tau>0$ \cite[Chapter 3]{Billingsley1999}.
This type of convergence is especially powerful as, if one can show Skorohod convergence (or processes convergence), denoted with $\xrightarrow{\mathcal{D}}$, then one can extend this to convergence in distribution of the confirmation times in a straightforward manner:
\begin{proposition}[Skorohod convergence implies convergence in distribution of the confirmation times]
\label{prop:Skorohod conv implies hitting}
    Consider an arbitrary process $(X_n(t))_{t\in[0,\tau]}$, and limiting process $(X(t))_{t\in[0,\tau]}$.
    Denote their respective confirmation times by $T_{X_n}$ and $T_{X}$. 
    Define $\mathcal{A}_\tau$,  the class of c\`adl\`ag functions that have a jump down at time $\tau$ and assume that $\prob(X\in \mathcal{A}_\tau)=0$
    and $\prob(T_X<\infty)=1$.
    If $(X_n(t))_{t\in[0,\tau]} \xrightarrow{\mathcal{D}} (X(t))_{t\in[0,\tau]}$ for arbitrary $\tau>0$, then 
    \begin{equation}
        T_{X_n} \xrightarrow{d}
    T_{X}.
    \end{equation}
\end{proposition}
\begin{proof}
We define the mapping $f^-:D[0,\tau] \to \mathbb{R}$ given by
\begin{equation}
    f^-\Big( \big(X(t)\big)_{t\in[0,\tau]}\Big) = \inf_{t\in[0,\tau]}\{X(t)\}.
\end{equation}
We aim to apply the continuous mapping theorem, but we note that $f^-$ is not necessarily continuous on $D[0,\tau]$. Specifically, the set of discontinuity points can be described by the set $\mathcal{A}_\tau$,  the class of c\`adl\`ag functions that have a jump down at time $\tau$.
However, as $\prob(X\in \mathcal{A}_\tau)=0$ by assumption,
we can apply the generalised continuous mapping theorem \cite[Theorem 3.4.4]{Whitt2002} and state that for arbitrary $\tau>0$
\begin{equation}
   \inf_{0\leq t\leq \tau}  X_n(t)
    \xrightarrow{d}
  \inf_{0\leq t\leq \tau}  X(t).
\end{equation}
This implies for fixed $\tau>0$ 
\begin{equation}
    \prob\Big( T_{X_n} > \tau 
    \Big)\to 
    \prob\Big( T_{X} > \tau 
    \Big).
\end{equation}
As $\tau$ can be arbitrarily large and $T_{X}<\infty$ by assumption, 
we recognise this as the tail distribution of $T_{X_n}$ and $T_{X}$ respectively. This implies convergence in distribution.
\end{proof}

This Proposition gives a clear road map on how we can prove Theorems \ref{thm:fluidlimit} and \ref{thm:diffusion_limit}, but showing Skorohod convergence of $\Q^{(n)}(t)/(Kn)$ and $\expec[X]\M^{(n)}(t)/(Kn)$ in the fluid limit and showing Skorohod convergence of $\Q^{(n)}(t)/(\sigma Kn\sqrt{n})$ and $\expec[X]\M^{(n)}(t)/(\sigma Kn\sqrt{n})$ is the main complication. 

The rest of this section is structured as follows. We start by providing a general proof scheme for all limiting statements in Section \ref{sec:proof_overview_limit}.
We then discuss every result separately, starting with the proofs for the fluid limit of the cumulative weight process in Section \ref{app:fluid proof workload}
and the transaction count process in Section \ref{app:fluid proof cust count}. This is followed by the proof for the diffusion limit of the cumulative weight process in Section \ref{app:dif proof workload} and the transaction count process in Section \ref{app:dif proof cust count}.

\subsection{Proof overview }
\label{sec:proof_overview_limit}
In the following we give a short outline of the ingredients needed to prove Theorem \ref{thm:fluidlimit} and \ref{thm:diffusion_limit}.
Furthermore, given these ingredients, we provide the proof of the aforementioned Theorems.
\paragraph{Fluid limit.}
Under specific assumptions, one  can often prove a fluid limit by showing pointwise convergence and then use the functional law of large numbers (c.f. \cite[Chapter 6]{Chen2001}). However, in our scenario, it is not clear if the functional law of large numbers applies, and we show Skorohod convergence manually. To do so, we introduce the following supplementary lemmas:

\begin{lemma}[Fluid limit for the cumulative weight process]
\label{lem:fluid conv workoad}
   Recall the scaled cumulative weight process $\Q^{(n)}$ from Equation \eqref{eq:defQn} 
   under the initial condition set in Equation \eqref{eq:fluid_initial_conditions}.
   Recall furthermore the limiting process $Y$ from Equation \eqref{eq:def-Y(t)}. Then for fixed $\tau>0$ and $0<t_1<\ldots<t_\ell<\tau$:
   \begin{itemize}
       \item[(a)] 
       \begin{equation}
       \label{eq:pw conv Q}
          \Big(\frac{1}{Kn} \Q^{(n)}(t_i)\Big)_{i=1}^\ell
          \xrightarrow{d} \Big(\Yl{y}(t_i)\Big)_{i=1}^\ell;
       \end{equation}
       \item[(b)]
       the process 
       $\frac{1}{Kn} \Q^{(n)}$
       is tight in the space $D[0,\tau]$.
   \end{itemize}
\end{lemma}
\begin{proof}
    See Section \ref{app:fluid proof workload}.
\end{proof}
A similar result can also be shown for the transaction count process:
\begin{lemma}[Fluid limit for the  transaction count process]
\label{lem:fluid conv cust count}
   Recall the scaled count process $\M^{(n)}$ from Equation \eqref{eq:defMn}
   under the initial condition set in Equation \eqref{eq:fluid_initial_conditions}. Recall furthermore
   the limiting process $Y$ from Equation \eqref{eq:def-Y(t)}. Then for fixed $\tau>0$ and $0<t_1<\ldots<t_\ell<\tau$:
   \begin{itemize}
       \item[(a)] 
       \begin{equation}
       \label{eq:pw conv M}
         \Big( \frac{\expec[X]}{Kn} \M^{(n)}(t_i)\Big)_{i=1}^\ell
          \xrightarrow{d} \Big(\Yl{y}(t_i)\Big)_{i=1}^\ell;
       \end{equation}
       \item[(b)]
       the process 
       $\frac{\expec[X]}{nK} \M^{(n)}$
       is tight in the space $D[0,\tau]$.
   \end{itemize}
\end{lemma}
\begin{proof}
    See Section \ref{app:fluid proof cust count}.
\end{proof}

Intuitively, both lemmas can be understood as follows: 
Firstly,
Equation \eqref{eq:pw conv Q} and Equation \eqref{eq:pw conv M} state that for fixed $t$, the weight  and transaction count processes look more like continuous time processes with downward jumps. The scaling in $n$ allows for a fast influx of transactions in a system where also the batch size scales in $n$. Normalising this by a factor $n^{-1}$ results in a process depicted in Equation \eqref{eq:def-Y(t)}. We refer to Figure \ref{fig:fluidlimits} for a visual representation. 

Secondly,
tightness is a technical detail both processes need to satisfy. While the formal definition is rather technical it is sufficient, for some given $t>0$, that one can provide a $\delta>0$ such that within an arbitrary interval $[t,t+\delta]$ both processes do not fluctuate too much.  

It turns out that the three aforementioned supporting results are sufficient to prove the main theorem in a straightforward manner.
\begin{proof}[Proof of Theorem \ref{thm:fluidlimit}]
    Fix $\tau>0$.
    Based on Lemmas \ref{lem:fluid conv workoad} and \ref{lem:fluid conv cust count} in combination with \cite[Theorem 13.1]{Billingsley1999} it follows that the convergence in Equation \eqref{eq:pw conv Q} and Equation \eqref{eq:pw conv M} can be extended to convergence in the space $D[0,\tau]$.  
    We also notice that $\Yl{y}(t)$ has jumps down at exponentially distributed times and therefore, for any $\tau>0$, $\prob(\Yl{y}(t)\in \mathcal{A}_{\tau})=0$.
    Finally, as we also assumed $\nu\expec[X]<\lambda K$, Proposition \ref{prop:Skorohod conv implies hitting} then concludes the proof.
\end{proof}

\paragraph{Diffusion limit.}
We implement a similar proof strategy for the diffusion limit. 
Such results are commonly shown via Donsker's theorem \cite[Chapter 4.3]{Whitt2002}, but, due to the added dependence in the BSQ, do not apply directly. Instead we show the following results:

\begin{lemma}[Diffusion limit for the cumulative weight process]
\label{lem:dif conv workoad}
   Recall the scaled cumulative weight process $\Q^{(n)}(t)$ from Equation \eqref{eq:defQn}
   under the initial condition set in Equation \eqref{eq:diffusion_initial_cond}. Recall furthermore
   Brownian motion $B_{y,c}(t)$ as defined in Theorem \ref{thm:diffusion_limit}. Let $\sigma=\sqrt{t\lambda}$, then for fixed $\tau>0$, 
       \begin{equation}
       \label{eq:dif conv Q}
          \Big(\frac{1}{\sigma Kn\sqrt{n}} \Q^{(n)}(nt)\Big)_{t\in[0,\tau]}
          \xrightarrow{\mathcal{D}} \Big(B_{y,-1}(t)\Big)_{t\in[0,\tau]}.
       \end{equation}
\end{lemma}
\begin{proof}
    See Section \ref{app:dif proof workload}.
\end{proof}
A similar result can also be shown for the transaction count process:
\begin{lemma}[Diffusion limit for the transaction count process]
\label{lem:dif conv cust count}
   Recall the scaled count process $\M^{(n)}(t)$ from Equation \eqref{eq:defMn} under the initial condition set in Equation \eqref{eq:diffusion_initial_cond}. Recall furthermore Brownian motion $B_{y,c}(t)$ as defined in Theorem \ref{thm:diffusion_limit}. Let $\sigma=\sqrt{t\lambda}$, then for fixed $\tau>0$, 
       \begin{equation}
       \label{eq:dif conv M}
         \Big( \frac{\expec[X]}{\sigma Kn\sqrt{n}} \M^{(n)}(nt)\Big)_{t\in[0,\tau]}
          \xrightarrow{\mathcal{D}} \Big(B_{y,-1}(t)\Big)_{t\in[0,\tau]}.
       \end{equation}
\end{lemma}
\begin{proof}
    See Section \ref{app:dif proof cust count}.
\end{proof}
In a similar fashion, we can now provide the proof of Theorem \ref{thm:diffusion_limit}.

\begin{proof}[Proof of Theorem \ref{thm:diffusion_limit}]
    As $B_{y,-1}(t)$ is a continuous process with a negative drift,
   the proof follows from Lemma \ref{lem:dif conv workoad} and \ref{lem:dif conv cust count} combined with Proposition \ref{prop:Skorohod conv implies hitting}.
\end{proof}
It therefore suffices to show Lemma \ref{lem:fluid conv workoad}--\ref{lem:dif conv cust count} which is done for each case in the following sections.

\subsection{Proof of Lemma \ref{lem:fluid conv workoad} - Fluid limit for the cumulative weight process}
\label{app:fluid proof workload}
Let $\{V_i\}_{i\geq 1}$ be the total number of transactions removed from the system after the $i$-th block mining event, under the convention that $V_0=m$. Define process 
\begin{equation}
\label{eq:defB}
J_k(\ell) = \sum_{i=V_k+1}^{V_k+\ell} X_i
\quad \text{ and } \quad
    B_k(x) = \text{argmax}_\ell
    \Big\{ J_k(\ell) < x\Big\},
\end{equation}
or in other words, the maximum number of transactions to arrive such that the cumulative sum does not exceed $x$. Let then 
\begin{equation}
\label{eq:defUk}
    U_k(x)=
    x-\sum_{i=V_k+1}^{V_k+B_k(x)}X_i.
\end{equation}
Intuitively $U_k(x)$ represents the empty space in the $k$-th block of size $x$ that could not be filled anymore by the following transaction, as including this transaction causes the exceeding of the maximal block size. Note, by this definition, that in case $\Q(t)$ or $\M(t)$ passes the $x$-axis, $U_k(x)$ is determined based on transactions that are for one part in the current BSQ and for another part from \textit{future transactions}, i.e. transactions in $\{X_i\}_{i\geq1}$ that are not in the BSQ yet. As we are only interested in the BSQ before it hits 0 the first time, this effect does not play a role in the rest of the analysis.

We now continue to the proof of Lemma \ref{lem:fluid conv workoad}.

\begin{proof}[Proof of Lemma \ref{lem:fluid conv workoad} (a)]
Recall $Q^{(n)}(t)$ from Equation \eqref{eq:defQn}.
Fix $t>0$,
    then we start by showing $\Q^{(n)}(t)/(Kn) \xrightarrow{d} \Yl{y}(t)$ pointwise in distribution,
    \begin{equation}
        \frac{1}{Kn}\Q^{(n)}(t) = 
         \underbrace{\frac{1}{Kn}\Bigg(\sum^{m(n)}_{h=1}X_h+\sum^{m(n)+A_{\nu n}(t)}_{i=m(n)+1} X_i -
        nKD_\lambda(t)\Bigg)}_{A_I}+
        \underbrace{\frac{1}{Kn}\sum_{j=1}^{D_\lambda(t)} U_j(nK)}_{A_{II}}.
    \end{equation}
    In the following we show pointwise convergence of $A_I$ and $A_{II}$ separately.
    \paragraph{Pointwise convergence of $A_I$:}
    We use generating functions  to show pointwise convergence of $A_I$, 
    \begin{equation}
    \begin{aligned}
        \expec[z^{A_I}]&=
        \expec[z^{\frac{1}{Kn}\big(\sum^{m(n)}_{h=1}X_h +\sum^{m(n)+A_{\nu n}(t)}_{i=m(n)+1}X_i - nKD_\lambda(t) \big)}]
        \\&= \expec[z^{\frac{1}{Kn}\sum^{m(n)}_{h=1}X_h}] \expec[z ^{ \frac{1}{Kn}  \sum^{A_{\nu n}(t)}_{i=1}X_i}]
        \expec[z ^{-D_\lambda(t)}].
    \end{aligned}
    \end{equation}
For the first generating function, we find by  \eqref{eq:fluid_initial_conditions}
\begin{equation}
    \expec[z^{\frac{1}{Kn}\sum^{m(n)}_{h=1}X_h}] 
    \to z^{\frac{yK}{K}}=z^y
\end{equation}
For the second expectation, by the tower rule, we can write
 \begin{equation}
 \label{eq:pointwise conv (A)}
      \expec[z ^{ \frac{1}{Kn}  \sum^{A_{\nu n}(t)}_{i=1}X_i}] = 
    \expec\Big[ 
    \expec[z ^{ \frac{1}{Kn}X}]^{A_{\nu n}(t)}
    \Big] = 
    \e^{-\nu nt (1-  \expec[z ^{ \frac{1}{Kn}X}])) }.
 \end{equation}
For $z\in(0,1)$ fixed and for $n$ large, we observe that $z^{1/(Kn)}\to 1$. In this regime, we can apply  a Tauberian expansion, see for example \cite[Theorem 8.1.6]{Bingham1987}, that states
\begin{equation}
\label{eq:tauberain_exapansion}
    \expec[z^{\frac{X}{Kn}}] = 
    1+\log(z^{\frac{1}{Kn}})\expec[X] + \frac{C}{(Kn)^2}(1+o(1)),
\end{equation}
where the error term is small for $n\to\infty$ and
$C$ is bounded, both uniform in $z\in[\varepsilon,1-\varepsilon]$ and in $K<1/\varepsilon$. Substituting this result we find 
\begin{equation}
\begin{aligned}
    \exp\Big\{-\nu nt (1-  \expec[z ^{ \frac{1}{Kn}X}])) \Big\}&=
      \exp\Big\{-\nu nt \Big(-\log(z^{\frac{1}{Kn}})\expec[X]-\frac{C}{n^2}(1+o(1)\Big)\Big\}
      \\&=z^{\nu t\expec[X]/K}(1+o(1)),
\end{aligned}
\end{equation}
uniformly in $z\in[\varepsilon,1-\varepsilon]$. According to a similar argument, we can derive  $\expec[z^{-D_\lambda(t)}]=\e^{-\lambda t(1-\frac{1}{z})}$, so that, by recognising the transform of $\Yl{q}(t)$,
\begin{equation}
\label{eq:(A)}
     \expec[z^{A_I}]=
     z^{y+\nu t\expec[X]}(1+o(1))
     \e^{-\lambda t(1-\frac{1}{z})}
    \to 
    \expec[z^{\Yl{y}(t)}]
     ,
\end{equation}
uniformly in $z\in[\varepsilon,1-\varepsilon]$, which shows pointwise convergence in distribution.
\paragraph{Pointwise convergence of $A_{II}$:}
As $D_{\lambda(t)}$ is independent of all $U_k(nK)$, it suffices to show that $\expec[U_k(nK)]$ converges in $n$. To do so, we observe that $U_k(x)$ corresponds to the \textit{spent time} of a delayed renewal process $N_k(t)$, that has inter-renewals times that have the same distribution as $X$, at time $x$. See also Figure \ref{fig:renewal_process} for details. 

The renewal process is delayed as the first transaction in the $k$-th block $X_{B_{k-1}(x)+1}$ is special as this one did not fit in the last block. 
We observe that, $X_{B_{k-1}(x)+1}$ is exactly the \textit{spent time} plus the \textit{residual time}, so $X_{B_{k-1}(x)+1}$ is almost surely bounded \cite[Theorem 1.18]{Mitov2014}.
Therefore $N_k(t)$ is a \textit{delayed renewal process} \cite[Chapter 5, Section 7A]{Karlin1975}. 

\begin{figure}
    \centering
    \tikzset{every picture/.style={line width=0.75pt}} 

\begin{tikzpicture}[x=0.75pt,y=0.75pt,yscale=-1,xscale=1]

\draw    (100.27,31.1) -- (100.93,225.1) ;
\draw    (78.93,211.1) -- (412.24,210.14) ;
\draw  [draw opacity=0] (101.6,210.53) .. controls (106.37,197.87) and (122.21,188.51) .. (141.06,188.4) .. controls (160.33,188.29) and (176.56,197.88) .. (181.08,210.91) -- (141.24,218.4) -- cycle ; \draw   (101.6,210.53) .. controls (106.37,197.87) and (122.21,188.51) .. (141.06,188.4) .. controls (160.33,188.29) and (176.56,197.88) .. (181.08,210.91) ;  
\draw  [draw opacity=0] (181.08,210.91) .. controls (183.26,198.07) and (193.66,188.3) .. (206.22,188.23) .. controls (218.76,188.16) and (229.26,197.78) .. (231.6,210.57) -- (206.38,215.93) -- cycle ; \draw   (181.08,210.91) .. controls (183.26,198.07) and (193.66,188.3) .. (206.22,188.23) .. controls (218.76,188.16) and (229.26,197.78) .. (231.6,210.57) ;  
\draw  [draw opacity=0] (276.16,210.7) .. controls (277.74,199.31) and (286.08,190.6) .. (296.17,190.55) .. controls (306.32,190.49) and (314.79,199.18) .. (316.45,210.64) -- (296.31,214.7) -- cycle ; \draw   (276.16,210.7) .. controls (277.74,199.31) and (286.08,190.6) .. (296.17,190.55) .. controls (306.32,190.49) and (314.79,199.18) .. (316.45,210.64) ;  
\draw  [draw opacity=0] (316.45,210.64) .. controls (321.21,198.7) and (336.64,189.9) .. (354.99,189.8) .. controls (373.5,189.69) and (389.14,198.47) .. (393.89,210.51) -- (355.15,218.33) -- cycle ; \draw   (316.45,210.64) .. controls (321.21,198.7) and (336.64,189.9) .. (354.99,189.8) .. controls (373.5,189.69) and (389.14,198.47) .. (393.89,210.51) ;  
\draw  [dash pattern={on 0.84pt off 2.51pt}]  (358.32,153.74) -- (358.32,209.74) ;
\draw  [dash pattern={on 0.84pt off 2.51pt}]  (316.72,107.1) -- (316.45,210.64) ;
\draw  [draw opacity=0] (316.27,155.32) .. controls (318.57,144.85) and (327.25,137.08) .. (337.58,137.12) .. controls (347.33,137.15) and (355.56,144.13) .. (358.32,153.74) -- (337.49,160.85) -- cycle ; \draw   (316.27,155.32) .. controls (318.57,144.85) and (327.25,137.08) .. (337.58,137.12) .. controls (347.33,137.15) and (355.56,144.13) .. (358.32,153.74) ;  
\draw  [draw opacity=0] (100.09,102.27) .. controls (106.69,90.56) and (152.54,81.63) .. (208.09,81.83) .. controls (268.13,82.05) and (316.77,92.86) .. (316.72,105.97) .. controls (316.72,106.08) and (316.72,106.2) .. (316.71,106.31) -- (208,105.57) -- cycle ; \draw   (100.09,102.27) .. controls (106.69,90.56) and (152.54,81.63) .. (208.09,81.83) .. controls (268.13,82.05) and (316.77,92.86) .. (316.72,105.97) .. controls (316.72,106.08) and (316.72,106.2) .. (316.71,106.31) ;  
\draw  [draw opacity=0][dash pattern={on 0.84pt off 2.51pt}] (316.71,106.31) .. controls (318.03,93.79) and (338.41,83.9) .. (363.32,84) .. controls (388.12,84.09) and (408.36,94.03) .. (409.75,106.49) -- (363.24,107.73) -- cycle ; \draw  [dash pattern={on 0.84pt off 2.51pt}] (316.71,106.31) .. controls (318.03,93.79) and (338.41,83.9) .. (363.32,84) .. controls (388.12,84.09) and (408.36,94.03) .. (409.75,106.49) ;  
\draw  [draw opacity=0] (316.71,106.31) .. controls (318.03,93.79) and (338.41,83.9) .. (363.32,84) .. controls (377.49,84.05) and (390.18,87.32) .. (398.7,92.43) -- (363.24,107.73) -- cycle ; \draw   (316.71,106.31) .. controls (318.03,93.79) and (338.41,83.9) .. (363.32,84) .. controls (377.49,84.05) and (390.18,87.32) .. (398.7,92.43) ;  

\draw (128.32,218.98) node [anchor=north west][inner sep=0.75pt]   [align=left] {$\displaystyle X_{1}$};
\draw (196.72,218.98) node [anchor=north west][inner sep=0.75pt]   [align=left] {$\displaystyle X_{2}$};
\draw (277.92,219.78) node [anchor=north west][inner sep=0.75pt]   [align=left] {$\displaystyle X_{B_{k}( x)}$};
\draw (332.32,219.78) node [anchor=north west][inner sep=0.75pt]   [align=left] {$\displaystyle X_{B_{k}( x) +1}$};
\draw (357.92,195.58) node [anchor=north west][inner sep=0.75pt]   [align=left] {$\displaystyle x$};
\draw (323.52,112.14) node [anchor=north west][inner sep=0.75pt]    {$U_{k}( x)$};
\draw (160.72,28.74) node [anchor=north west][inner sep=0.75pt]   [align=left] {Transactions in $ $\\block $\displaystyle k$};
\draw (328.72,33.54) node [anchor=north west][inner sep=0.75pt]   [align=left] {Transactions in\\ block $\displaystyle k+1$};
\draw (81.6,6.67) node [anchor=north west][inner sep=0.75pt]   [align=left] {$\displaystyle N_{k}( t)$};
\draw (416.27,198) node [anchor=north west][inner sep=0.75pt]   [align=left] {$\displaystyle t$};

\end{tikzpicture}
    \caption{Sketch of the renewal process $N_k(t)$ that is used to determine the distribution of $U_k(x)$.
    For a block size $x$, only $B_k(x)$ transactions can fit, leaving an open space of $U_k(x)$.
    }
    \label{fig:renewal_process}
\end{figure}
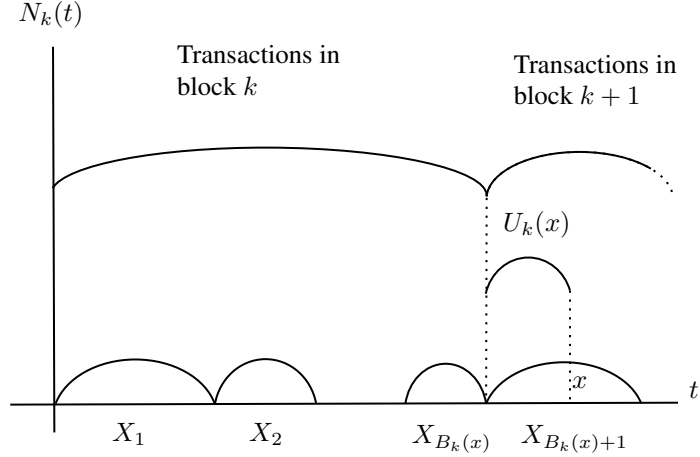

For a regular renewal process the scaling of the expected spent time is known,  \cite[Theorem 1.18]{Mitov2014}.
Moreover, many properties for the delayed renewal process are asymptotically the same as for the non-delayed renewal process \cite[Theorem 1.20]{Mitov2014}, 
inspection of the proof of \cite[Theorem 1.18]{Mitov2014} shows that
\begin{equation}
\label{eq:convU1}
    \lim_{x\to\infty }
    \expec[U_k(x)] =
   \frac{\expec[X^2]}{2\expec[X]}.
\end{equation}
As the blocksize is scaled to infinity, we find therefore that $U_k(nK)$ has bounded expectation for $n\to\infty$.

By the Markov inequality, one finds for any given $\varepsilon>0$ that
\begin{equation}
\label{eq:(B)}
    \prob\bigg(
    \frac{1}{Kn}\sum_{j=1}^{D_\lambda(t)} U_j(nK)>\varepsilon
    \bigg)
    \leq
    \frac{
    \expec\left[\sum_{j=1}^{D_\lambda(t)} U_j(nK)\right]
    }{\varepsilon Kn}
    =\frac{\expec[X^2]\expec[D_{\lambda}(t)]}{2\expec[X]\varepsilon Kn}(1+o(1))
    \to 0.
\end{equation}

Based on the results of Equation \eqref{eq:(A)} and Equation \eqref{eq:(B)}, we find by Slutsky's theorem \cite[Theorem 11.4]{Gut2012},
\begin{equation}
    \frac{1}{nK}\Q^{(n)}(t)
    \to \Yl{y}(t).
\end{equation}

\paragraph{Convergence of the finite dimensional distribution:}
Extending this to the finite dimensional distribution is now straightforward. Fix $\ell\in\mathbb{N}^+$ and $s_1,\ldots, s_\ell>0 $ and $0<t_1<\ldots<t_\ell<\tau$.
Then, by \eqref{eq:(B)} and Slutsky's theorem, the joint transform can be simplified to
\begin{equation}
\begin{aligned}
   & \expec[z^{\frac{1}{nK}
    (s_1Q^{(n)}(t_1)+\ldots+s_\ell Q^{(n)}(t_\ell))
    }] \\&\qquad= 
    \expec\Big[ z^{\ell \frac{1}{nK}\sum_{h=0}^{m(n)}X_h
    +\sum_{j=1}^\ell s_i\big[
    \frac{1}{Kn}\sum_{i=m(n)+1}^{m(n)+A_{\nu n}(t_j)}X_i
    -D_\lambda(t_j)
    +\frac{1}{Kn}\sum^{D_\lambda(t_j)}_{i=1} U_i(nK)]
    }\Big]
    \\& \qquad= 
    \expec\Big[ z^{\ell \frac{1}{nK}\sum_{h=0}^{m(n)}X_h
    +\sum_{j=1}^\ell s_i\big[
    \frac{1}{Kn}\sum_{i=m(n)+1}^{m(n)+A_{\nu n}(t_j)}X_i
    -D_\lambda(t_j)
    \big]
    }\Big](1+o(1)),
\end{aligned}
\end{equation}
where the last equality follows from \eqref{eq:(B)}.
Using that $A_{\nu n}(t)$ and $D_\lambda(t)$ are Poisson counting processes, that have stationary and independent increments, this can be rewritten to
\begin{equation}
\begin{aligned}
\label{eq:findemQ1}
     \expec\Big[ z\text{\textasciicircum}\ \Big\{&\ell \frac{1}{nK}\sum_{h=0}^{m(n)}X_h
   + \frac{(s_1+\ldots+\ldots s_\ell)}{Kn}\sum_{i=1}^{A_{\nu n}(t_1)}
    X_i^{(1)}+
    \ldots
    +\frac{s_\ell}{Kn} \sum_{i=1}^{A_{\nu n}(t_\ell-t_{\ell-1})}X_i^{(\ell)}
    \\&
    -(s_1+\ldots +s_\ell)D_\lambda(t_1)
    -\ldots -
    s_\ell D_\lambda(t_\ell-t_{\ell-1})\Big\}\Big](1+o(1)),
\end{aligned}
\end{equation}
where $X_i^{(j)}$ are i.i.d. copies of $X$. As now all terms are independent, we can use \eqref{eq:(A)} to show that \eqref{eq:findemQ1} converges for $n\to\infty$ to
\begin{equation}
\label{eq:findemQ2}
    z^{\ell y +\nu\expec[X](
    t_1(s_1+\ldots+ s_\ell) + \ldots +
    (t_\ell-t_{\ell-1})s_\ell )
    }\e^{\lambda(1-z)\big( t_1(s_1+\ldots +s_\ell)+\ldots +(t_\ell-t_{\ell-1})s_\ell\big)}.
\end{equation}
This term can be similarly rearranged, again using that the CL model also has stationary and independent increments, where one finds that \eqref{eq:findemQ2} is equal to the joint transform of  $(\Yl{q}(t_i))_{i=1}^\ell$
\end{proof}

Then, it remains to show tightness, which is a technical necessity for the fluid limit. Proving tightness is rather technical and therefore presented in Appendix \ref{app:tightness}.
\begin{proof}[Proof of Lemma \ref{lem:fluid conv workoad} (b)]
See Appendix \ref{app:tightness}.

\end{proof}

\subsection{Proof of Lemma \ref{lem:fluid conv cust count} - Fluid limit for the transaction count process}
\label{app:fluid proof cust count}
Recall $V_k$ as the total number of transactions that left the system after the $k$-th departure event.
Recall the random variable $B_k(x)$ from Equation \eqref{eq:defB} that denotes how many transactions one can maximally fit in the $k$-th block of size $x$. We then continue according to the same approach as in the analysis for the weight.

\begin{proof}[Proof of Lemma \ref{lem:fluid conv cust count} (a)]
Recall $\M^{(n)}(t)$ from Equation \eqref{eq:defMn} and the weight sequence of transactions given by$\{X_i\}_{i\geq 1}$, then
\begin{equation}
    \begin{aligned}
         \frac{\expec[X]}{Kn}\M^{(n)}(t) &=    
          \frac{\expec[X]m(n)}{Kn}+
          \frac{\expec[X]}{Kn}A_{\nu n}(t)-  \frac{\expec[X]}{Kn}\sum_{j=1}^{D_\lambda(t)} B_j(Kn)\\&=
        \underbrace{  \frac{\expec[X]m(n)}{Kn} +\frac{\expec[X]}{Kn}\Big( A_{\nu n}(t)-\frac{KnD_\lambda(t)}{\expec[X]}\Big)}_{(C)}+
         \underbrace{  \frac{\expec[X]}{Kn}\sum_{j=1}^{D_\lambda(t)} \Big(B_j(Kn)-\frac{Kn}{\expec[X]}\Big)}_{(D)}
         ,
    \end{aligned}
\end{equation}
where we start with showing pointwise convergence for fixed $t$. Pointwise convergence for $(C)$ is similar to pointwise convergence of $(A)$, in Equation \eqref{eq:pointwise conv (A)}. Hence, by the same argumentation and the initial condition in \eqref{eq:diffusion_initial_cond}
\begin{equation}
\label{eq:(C)}
        (C)= \frac{\expec[X]m(n)}{Kn}+
      \frac{\expec[X]}{Kn}A_{\nu n}(t) -D_\lambda(t)
      \to 
      \Yl{y}(t).
\end{equation}
In the following, we mainly focus on pointwise convergence of $(D)$. For this, we notice that for $B_k(x)$
\begin{equation}
\label{eq:trans_count_B_to_sum}
\{B_k(x)>\ell\}=
\bigg\{\text{argmax}_m
\Big\{\sum_{i=1}^{m} X_i < x\Big\}>\ell
\bigg\} = 
 \Big\{ \sum^{\ell+1}_{i=1} X_i < x\Big\}.
\end{equation}
Fix $\varepsilon>0$, then we show

\begin{equation}
\label{eq:(D)}
    \begin{aligned}
        \prob\bigg( 
    \frac{1}{n}\Big(B_k(x)-& \frac{Kn}{\expec[X]}
    \Big)>\varepsilon
        \bigg)= 
        \prob\bigg(
        \sum^{n\varepsilon+Kn/\expec[X]+1}_{i=1}X_i
        <Kn\bigg)\\&=
        \prob\bigg(
        \frac{1}{n\varepsilon+Kn/\expec[X]+1} \sum^{n\varepsilon+Kn/\expec[X]+1}_{i=1}X_i
        <\frac{\expec[X]}{1+\expec[X]/K(1/n+\varepsilon)}\bigg)
        \\&\to 0,
    \end{aligned}
\end{equation}
based on the law of large numbers, where we observe that $\expec[X]/(1+\expec[X]/K(1/n+\varepsilon))<\expec[X]$. 
Note also that by the observation from the proof of Lemma \ref{lem:fluid conv workoad}, $X_1$ is special, as it is the sum of a \textit{spent time} and a \textit{residual time}. However, in the limit, we observe that the expected size of this special transaction is bounded according to Equation \eqref{eq:convU1}. 
Therefore, the result still holds.
Similarly, one can also show
\begin{equation}
\label{eq:(D)_2}
        \prob\bigg( 
    \frac{1}{n} \Big(B_k(x)- \frac{Kn}{\expec[X]}
    \Big)<-\varepsilon
        \bigg)\to 0, 
\end{equation}
which implies that $(B_1-Kn/\expec[X])/n\to 0$
in probability. 

Based on the results of Equation \eqref{eq:(C)} and Equations \eqref{eq:(D)} and \eqref{eq:(D)_2}, we find by Slutsky's theorem \cite[Theorem 11.4]{Gut2012}, 
\begin{equation}
\label{eq:conv1dmM}
      \frac{\expec[X]}{Kn}\M^{(n)}(t) 
      \to \Yl{y}(t).
\end{equation}
Extending the result of \eqref{eq:conv1dmM} is equivalent to that of $Q^{(n)}(t)$ and applying similar techniques as in \eqref{eq:findemQ1} and \eqref{eq:findemQ2} yields the desired result.
\end{proof}

Finally we also show tightness:
\begin{proof}[Proof of Lemma \ref{lem:fluid conv cust count} (b)]
    The proof is a special case of the proof of Lemma \ref{lem:fluid conv workoad} (b) with $X\equiv 1$. 
\end{proof}

\subsection{Proof of Lemma \ref{lem:dif conv workoad} - Diffusion limit for the cumulative weight process}
\label{app:dif proof workload}
We continue with the proof for the diffusion limit for the cumulative weight process.

\begin{proof}[Proof of Lemma \ref{lem:dif conv workoad}]
Recall $Q^{(n)}(t)$ from Equation \eqref{eq:defQn} and its initial conditions in Equation \eqref{eq:diffusion_initial_cond}. We then start with the convergence for fixed $t$. Let $\sigma=\sqrt{t\lambda}$ and $n=k^2\sigma^2/(\lambda K - \nu\expec[X])$, then

\begin{equation}
   \frac{ Q^{(n)}(nt)}{K\sigma n\sqrt{n}} =
   \frac{ 1}{K\sigma n\sqrt{n}}
   \Big[
    \sum_{h=1}^{m(n)}
    X_h
    +\sum_{i=m(n)+1}^{m(n)+A_{n\nu}(nt)}
    X_i
    -\sum_{j=1}^{D_\lambda(nt)}(nK - U_j(nK))
   \Big]
\end{equation}
Observing that $nK-U_j(nK)$ can be stochastically bounded by an a.s. finite random variable, a similar argument to \eqref{eq:(B)} can be applied. Combining this with \eqref{eq:diffusion_initial_cond}, we find 
\begin{equation}
     \frac{ 1}{K\sigma n\sqrt{n}}
     \sum_{j=1}^{D_\lambda(nt)}U_j(nK) \to 0 \qquad
     \text{ and }
     \qquad 
     \frac{ 1}{K\sigma n\sqrt{n}}\sum_{h=1}^{m(n)}
    X_h \to y.
\end{equation}
Therefore, by Slutsky's theorem and the stationarity of the Poisson process, it is sufficient to consider the scaling of 
\begin{equation}
\begin{aligned}
    \frac{ 1}{K\sigma n\sqrt{n}}
   \Big[
    \sum_{i=1}^{A_{n\nu}(nt)}
    X_i
    -D_\lambda(nt)nK 
    \Big]&=
    \frac{ 1}{K\sigma n\sqrt{n}}
   \Big[
    \sum_{i=1}^{A_{n\nu}(nt)}
    X_i-\nu\expec[X]n^2t\Big] + \frac{\sqrt{n}}{K\sigma}\nu\expec[X]t\\&
    -
     \frac{ 1}{K\sigma n\sqrt{n}}
     \Big[ 
     D_\lambda(nt)nK-\lambda K n^2t
     \Big] - \frac{\sqrt{n}}{{\sigma}}\lambda t.
\end{aligned} 
\end{equation}
We notice that the first term converges to 0 by the central limit theorem. Furthermore, by applying Donsker's theorem for the second expression and recalling the definition of $n$, we find by Slutsky's theorem
\begin{equation}
      \frac{ Q^{(n)}(nt)}{K\sigma n\sqrt{n}}
      \xrightarrow{\mathcal{D}} 
      y+0-B_{0,0}(t)-t \equiv B_{y,-1}(t),
\end{equation}
which is a Brownian motion with starting position $y$ and drift $-1$.
\end{proof}

\subsection{Proof of Lemma \ref{lem:dif conv cust count} - Diffusion limit for the transaction count process}
\label{app:dif proof cust count}

\begin{proof}[Proof of Lemma \ref{lem:dif conv cust count}]
Recall $M^{(n)}(t)$ from Equation \eqref{eq:defMn} and its initial conditions in Equation \eqref{eq:diffusion_initial_cond}. We then start with the convergence for fixed $t$. Let $\sigma=\sqrt{t\lambda}$ and $n=K^2\sigma^2/(\lambda K - \nu\expec[X])$, then

\begin{equation}
\label{eq:diffusion_trans_count_1}
\begin{aligned}
   \frac{ \expec[X] M^{(n)}(nt)}{K\sigma n\sqrt{n}} &=
   \frac{\expec[X]}{K\sigma n\sqrt{n}}
   \Big[
    m(n) + A_{n\nu}(nt) - \sum_{j=1}^{D_{\lambda}(nt)} B_{j}(Kn)
   \Big]
    \\&=
   \frac{\expec[X]}{K\sigma n\sqrt{n}}
   \Big[
    m(n) + A_{n\nu}(nt) - 
    D_{\lambda}(nt) \frac{Kn}{\expec[X]}
    -\sum_{j=1}^{D_{\lambda}(nt)} \Big(B_{j}(Kn)-\frac{Kn}{\expec[X]}\Big)
   \Big].
\end{aligned}
\end{equation}
Convergence for the $m(n)$, $A_{n\nu}(nt)$ and $D_\lambda(nt)$ term follow from a similar argument as presented for the diffusion limit of the cumulative weight process and it therefore suffices to derive the diffusion limit of the sum in \eqref{eq:diffusion_trans_count_1}.
However, contrary to the diffusion limit for the cumulative weight process where $U_j(Kn)$ could be bounded by an a.s. finite random variable, the fluctuations from the $B_j(Kn)-Kn/\expec[X]$ term introduce an error in $n$. It is therefore not straightforward that this term disappears in the limit for $n\to\infty$.
For fixed $j$, we observe, by \eqref{eq:trans_count_B_to_sum}
\begin{equation}
\begin{aligned}
    \prob\Big(
    \frac{B_j(Kn)- Kn/\expec[X]}{\sqrt{n}}
    >x\Big) &=
      \prob(
    B_j(Kn)> x\sqrt{n} +Kn/\expec[X])
    \\&= 
    \prob\Big( \sum_{i=1}^{x\sqrt{n} +Kn/\expec[X]+1}
    X_i < Kn\Big).
\end{aligned}    
\end{equation}
By simplifying the expression to a sum of i.i.d. random variables, we can apply a central limit result to continue 
\begin{equation}
\begin{aligned}
    &= 
     \prob\Big( \frac{\sum_{i=1}^{x\sqrt{n} +Kn/\expec[X]+1}
    X_i -\expec[X](x\sqrt{n}+Kn/\expec[X]+1)}{\sigma
    \sqrt{x\sqrt{n} +Kn/\expec[X]+1}} < \frac{Kn-\expec[X](x\sqrt{n}+Kn/\expec[X]+1)}{\sigma
    \sqrt{x\sqrt{n} +Kn/\expec[X]+1}} \Big)
    \\&\to 
    \prob\Big(\mathcal{N}(0,1) < -\frac{\expec[X]x}{\sigma\sqrt{K/\expec[X]}}\Big)
    =\prob(\mathcal{N}(0,C) > x), 
\end{aligned}
\end{equation}
for $C= \sigma^2K/\expec[X]$. For all $j$, we then have
\begin{equation}
    \frac{B_j(Kn) - Kn/\expec[X]}{\sqrt{n}}
    \to 
    \mathcal{N}(0,C).
\end{equation}
Therefore, we find for the sum in \eqref{eq:diffusion_trans_count_1},
\begin{equation}
\begin{aligned}
    \expec\Big[\e^{-\frac{s}{n}\sum_{j=1}^{D_\lambda(nt)} \frac{B_j(Kn)-Kn/\expec[X]}{\sqrt{n}}}\Big]&=
    \expec\Big[ 
    \expec\big[\e^{-\frac{s}{n}\frac{B_j(Kn)-Kn/\expec[X]}{\sqrt{n}}}\big]^{D_{\lambda}(nt)}
    \Big]\\&=
    \expec\big[\e^{-\frac{Cs^2}{2n^2}D_\lambda(nt)}(1+o(1)\big]
    \\&= \e^{-\lambda nt(1-\e^{-\frac{Cs^2}{2n^2}})}(1+o(1))
    \\&=  \e^{-\frac{\lambda t Cs^2}{2n}}(1+o(1)).
\end{aligned}
\end{equation}
Based on this result, we find for \eqref{eq:diffusion_trans_count_1}, based on Donsker's theorem
\begin{equation}
      \frac{ \expec[X] M^{(n)}(nt)}{K\sigma n\sqrt{n}}
      \xrightarrow{\mathcal{D}}
      y+0+B_{0,0}(t) -t\equiv 
      B_{y,-1}(t).
\end{equation}
This concludes the result.
\end{proof}

\section{Numerical results}
\label{sec:results}
We start by giving a visual overview of the  confirmation time density in minutes in Figure \ref{fig:conftimedens}. Here, we fixed the value of $c$ and plotted the density using Equation \eqref{eq:densityut} for $y=0$ and $y=5$.

In Figure \ref{fig:conftimedens}, one notices a saw-tooth like shape, meaning that there are points with almost zero density. These are located at $n/c$, $n=1,2,\ldots $ If one wants to confirm at these time instances or immediately thereafter, instead of one, at least two blocks need to be mined almost simultaneously, which has negligible density. Additionally, in Figure \ref{fig:Stadjeplots}, we provide the tail distribution of $N_y$ for $y=0$ and $y=3$.

\begin{figure}[h]
\centering
\begin{subfigure}{.5\textwidth}
  \centering
  \includegraphics[width=1\linewidth]{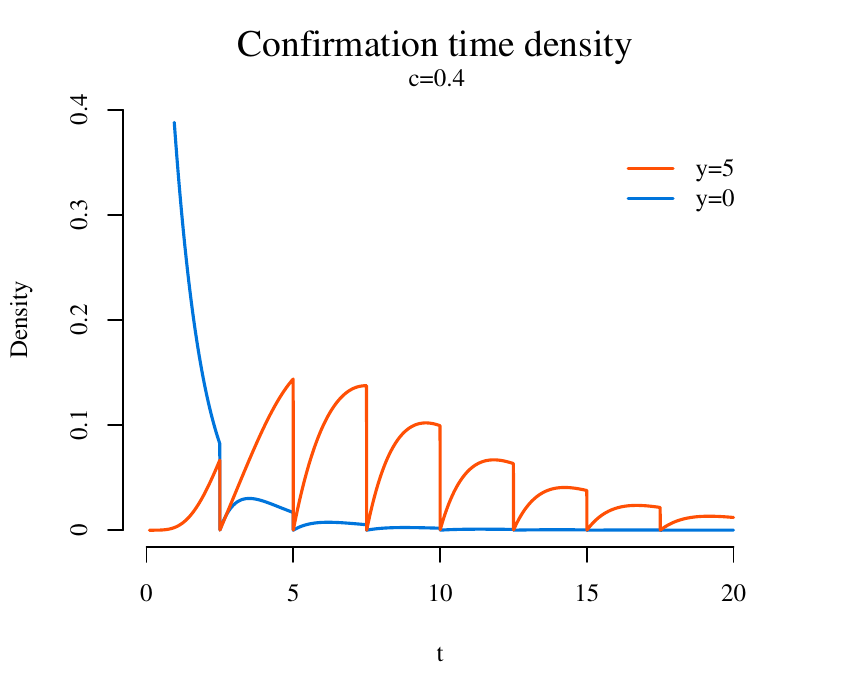}
  \caption{ }
  \label{fig:denssub1}
\end{subfigure}%
\begin{subfigure}{.5\textwidth}
  \centering
  \includegraphics[width=1\linewidth]{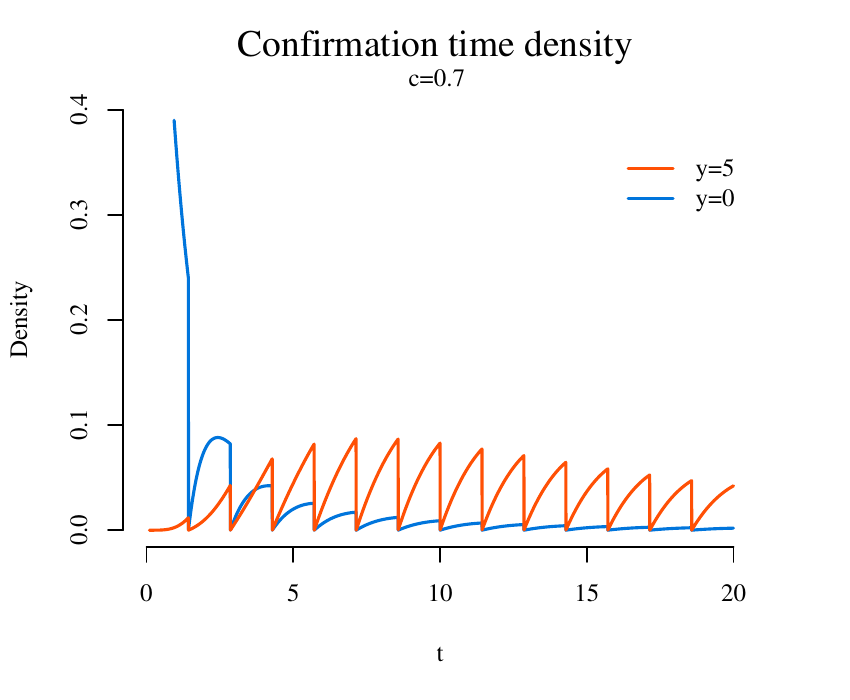}
  \caption{ }
  \label{fig:denssub2}
\end{subfigure}
\caption{The density of  the  confirmation times, c.f.  Equation \eqref{eq:densityut}, for $c=0.4$ (a) and $c=0.7$ (b), and for $y=0$ and $y=5$.}
\label{fig:conftimedens}
\end{figure}

\begin{figure}[h]
\centering
\begin{subfigure}{.45\textwidth}
  \centering
  \includegraphics[width=1\linewidth]{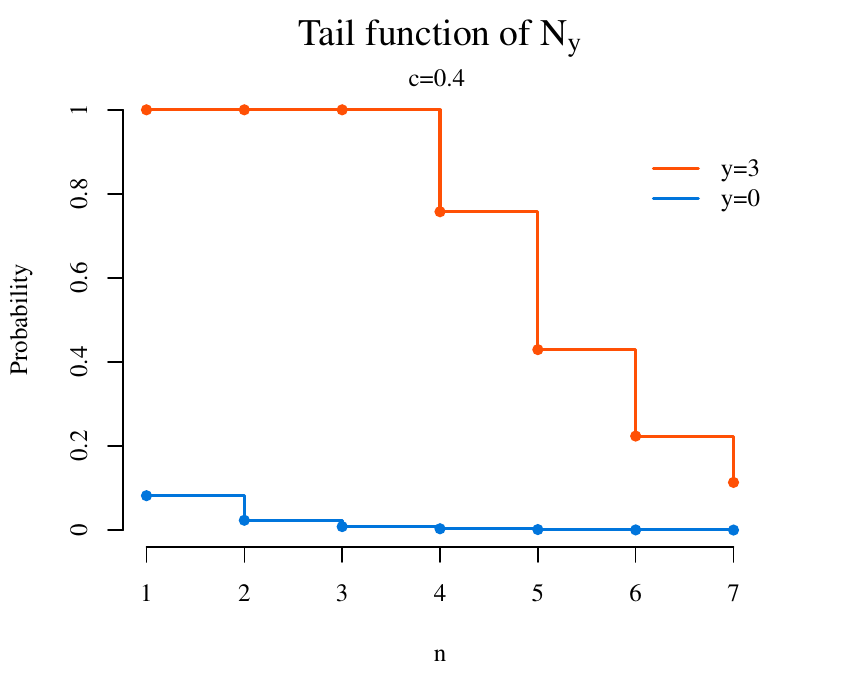}
  \caption{ }
  \label{fig:stadjec=0.4}
\end{subfigure}%
\begin{subfigure}{.45\textwidth}
  \centering
  \includegraphics[width=1\linewidth]{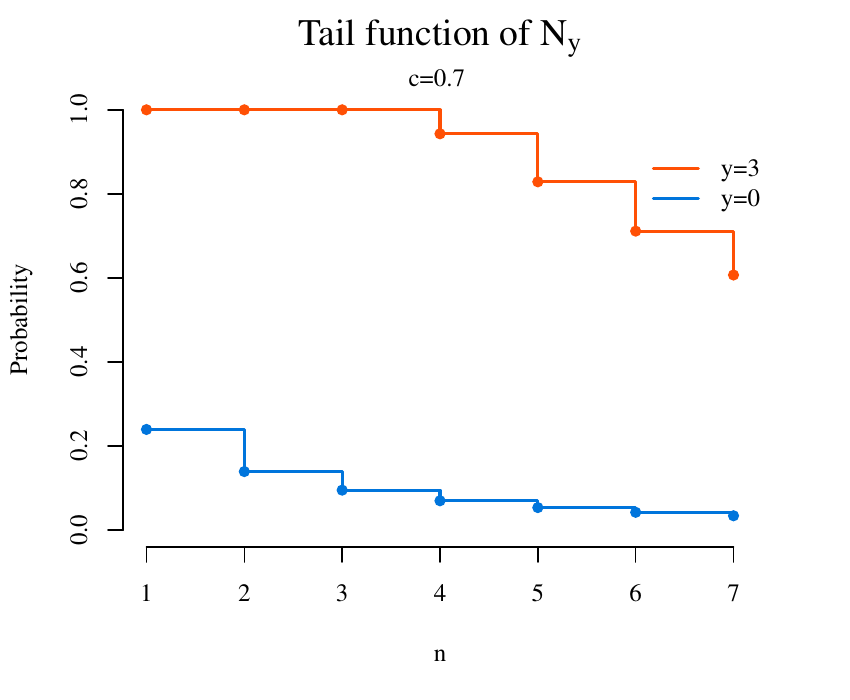}
  \caption{ }
  \label{fig:stadjec=0.7}
\end{subfigure}
\caption{Tail distribution of the number of blocks to confirmation,  c.f. Equation \eqref{eq:stadjey=y}, for different values of $y$.
Note that $\prob[N_y>n]$ equals 1 in case $n< y+1$.}
\label{fig:Stadjeplots}
\end{figure}

Next, we show some numerical results for the mean confirmation time in Figure \ref{fig:plotsEt}. Note that, due to  Theorem \ref{lem:T=N}, it suffices to only consider $\mathbb{E}[T_y]$, as $\mathbb{E}[N_y]=\mathbb{E}[T_y]$. For the computation of the expected confirmation times (as a function of the starting position $y$), for $c=0.2$ and $c=0.9$, we use Equation \eqref{eq:E(T)alg} in combination with Equations \eqref{eq:E(T)} and \eqref{eq:E(t)y<1}.

\begin{figure}[ht]
\centering
  \includegraphics[scale=.5]{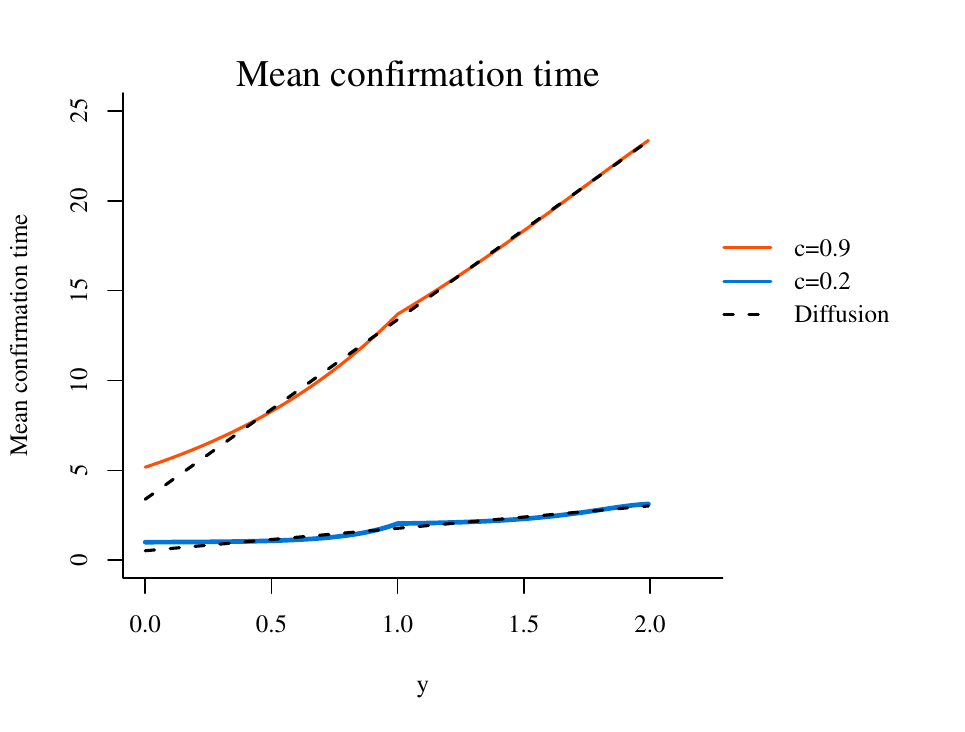}
\caption{The mean confirmation time, c.f. Equations \eqref{eq:E(T)alg} and \eqref{eq:E(t)y<1}, as a function  of the starting position $y$, for $c=0.2$ and $c=0.9$, is depicted with the coloured lines. The corresponding  diffusion approximation, c.f. Equation \eqref{eq:E(t)approx}, is depicted with the dashed black line.}
\label{fig:plotsEt}
\end{figure}

Based on the results in Figure \ref{fig:plotsEt}, we observe that for $y>1$,  the expected confirmation time seems to increase linearly in $y$. 
Intuition for this follows from Remark \ref{rem:diffusionapprox}
 and we overlay the linear approximation arising from the diffusion limit in \eqref{eq:E(t)approx} in Figure \ref{fig:plotsEt}, which seems to represent $\mathbb{E}[T_y]$ well.
We refer to Table \ref{tab:E(T)results}, for a numerical comparison of the theoretical mean and the approximation. Note that for $y<1$, the expression can be computed directly by Equation \eqref{eq:E(t)y<1}, and no approximation is needed.

\begin{table}[ht]
    \centering
    \begin{tabular}{c|cc}
        $c=0.3$ &  T & A\\
        \hline
$y=0.5$&	1.226	&	1.288	\\
$y=1.0$	&2.195	&	2.002	\\
$y=1.5$	&2.655	&	2.716	\\
$y=2.5$	&4.130	&	4.145	\\
$y=3.5$	&5.578	&	5.574	\\
$y=4.5$	&7.006	&	7.002	\\
$y=5.5$	&8.432	&	8.431	\\
$y=6.5$	&9.859	&	9.859	\\
    \end{tabular}
    \quad
    \begin{tabular}{c|cc}
        $c=0.6$ &  T & A\\
        \hline
$y=0.5$	&	2.104	&	2.155	\\
$y=1.0$	&	3.540		&	3.405	\\
$y=1.5$	&	4.624	&	4.655	\\
$y=2.5$	&	7.152	&	7.155	\\
$y=3.5$	&	9.656	&	9.655	\\
$y=4.5$	&	12.156	&	12.155	\\
$y=5.5$	&	14.655	&	14.655	\\
$y=6.5$	&	17.155	&	17.155	\\
    \end{tabular}
    \quad
    \begin{tabular}{c|cc}
        $c=0.9$ &  T &A\\
        \hline
$y=0.5$	&	8.283	&8.392	\\
$y=1.0$	&	13.694	&13.392	\\
$y=1.5$	&	18.335	&18.392	\\
$y=2.5$	&	28.388	&28.392	\\
$y=3.5$	&	38.393	&38.392	\\
$y=4.5$	&	48.392	&48.392	\\
$y=5.5$	&	58.392	&58.392	\\
$y=6.5$	&	68.392	&68.392	\\
    \end{tabular}
    \caption{Comparisons of the theoretical means (T) (c.f.  Equation \eqref{eq:E(T)alg}) to the approximated means (A) (c.f. Equation \eqref{eq:E(t)approx}).}
    \label{tab:E(T)results}
\end{table}
We observe that the approximated means are close to the theoretical ones, especially for $y$ large and can therefore be used as a suitable alternative in case a fast evaluation is needed.

\section{Model validation and prediction - a case study}
\label{sec:case_study}

In this section, we validate our model-based approach to real-world Bitcoin data. To this end, we consider Bitcoin mempool data collected\footnote{Data extracted from a node, maintained by \cite{hoenicke}} over approximately 2 years (19-12-2018 to 15-12-2020). 
This data is used to show the following results:
\begin{enumerate}
    \item First, we \textit{validate the model}. We propose methods for estimating the model parameters $c$ and $y$ from the mempool data. We use this to compare
    the mean confirmation time obtained from the data to the mean confirmation times of the CL model and its heavy-traffic approximation, the BM. The main results are shown in Table \ref{tab:CS1} in Section \ref{sec:cs_validation}.
    \item Secondly, we analyse \textit{the performance of the model}. We compare the model-based method to the performance of a purely data-driven method that closely resembles fee estimators that are used in practice. The main results are shown in Table \ref{tab:CS2} in Section \ref{sec:cs_performance}.
\end{enumerate}
The data and the code used for the data analysis, figures and tables of this section can be found in \cite{code}.

The remainder of this section is structured as follows. 
In Section \ref{sec:data_extraction} we detail the extraction procedure for transaction confirmation times from the mempool data.
Such a procedure deserves extra care as it is not obvious how to extract independent confirmation times with the same parameters ($c$, $y$), which is paramount for a proper and interpretable validation.
We then use the mempool data and extracted confirmation times to support the results in Section \ref{subsec:1modelval}, and the transaction data to support the results of Section \ref{subsec:1modelperf}. This is detailed in Sections \ref{sec:cs_validation} and \ref{sec:cs_performance} respectively.
\subsection{Data extraction}
\label{sec:data_extraction}
The data captures the cumulative weight of mempool transactions clustered by different fee densities\footnote{These ranges are between two sequential values in: [0, 1, 2, 3, 4, 5, 6, 7, 8, 10, 12, 14, 17, 20, 25, 30, 40, 50, 60, 70, 80, 100, 120, 140, 170, 200, 250, 300, 400, 500, 600, 700, 800, 1000, 1200, 1400, 1700, 2000, 2500, 3000, 4000,5000, 6000, 7000, 8000, 10000,$\infty$]. } at discrete time epochs $\{t_j\}_{j=0}^N$ that differ on average by 1 minute\footnote{The time between two sequential epochs is, in 97.4\% of cases between 54 and 66 seconds.}.
These different fee density levels are also called buckets. For any of the aforementioned fee levels $\phi$, one can define the cumulative mempool weight process, i.e. the sum of the weights of all transactions currently in the mempool over time. This allows us to observe the cumulative mempool weight of transactions with a fee density exceeding $\phi$ corresponding to the lower end of a priority bucket over time. Denote this weight at  time $t$ by $\hat Y_\phi(t)$
such that all our observations are denoted by $\{\hat Y_\phi(t_j)\}_{j=0}^N$, and let $X^\phi_j = \hat Y_\phi(t_j) - \hat Y_\phi(t_{j-1})$ denote the increments. When the cumulative mempool weight is zero, it implies that the transaction with fee $\phi$ is confirmed.  

We scale the time unit such that the arrival rate for blocks of the CL model equals 1 unit of time. Since, in the Bitcoin network, the block inter-arrival time has a mean of approximately 10 minutes, it follows that one unit of time corresponds to 10 minutes, and thus we observe the process approximately every 0.1 time unit. Furthermore the cumulative mempool weight is considered in vMB, so that 1 vMB corresponds to the block size.

The main difficulty comes from the discreteness of data. If the mempool became empty at some point, it is unlikely that it remains empty at the next time epoch where we observe the cumulative mempool weight, c.f. Figure \ref{fig:discrete-problem-example}. 
To overcome this, we define a 
function \texttt{Get\_confirmation\_time} (also abbreviated by function $\texttt{G}(\{Y_\phi(t_j)\}_{j=j_0}^{j_{\max}},\hat c_\phi)$) that extracts a confirmation time from the mempool data $\{Y_\phi(t_j)\}_{j=j_0}^{j_{\max}}$ when given an estimated slope $\hat c_\phi$.
It uses a criterion for the mempool being empty between time epochs, called the \textit{$\varepsilon$-criterion}. Intuitively, the $\varepsilon$-criterion induces a confirmation time, only when,  in the next epoch, a decrease of the mempool is observed and at that time the mempool size is at most $B-\varepsilon$, see Algorithm \ref{alg:getconftime} line 3 for an implementation for $\varepsilon=\hat c_\phi/10$.

\begin{algorithm}[ht]
\caption{\texttt{Get\_confirmation\_time}: \texttt{G}$(\{Y_\phi(t_j)\}_{j=j_0}^{j_{\max}},\hat{c}_\phi)$ }
\begin{algorithmic}[1]
\State Input: $\{Y_\phi(t_j)\}_{j=j_0}^{j_{\max}}$, $\hat{c}_\phi$
\For{$j$ in $\{j_0,\ldots ,j_{\max}\}$}
\If{$\hat{Y}_\phi(t_{j-1})< B -\varepsilon,\
    \hat{Y}_\phi(t_{j}) < \hat{Y}_\phi(t_{j-1})$}
\State $j_{end} = j$
\State break
\EndIf 
\EndFor
\If{$j_{end}>j_0$}
\State return($j_{end}-j_0$)
\Else 
\State return( No confirmation time)
\EndIf 
\end{algorithmic}
\label{alg:getconftime}
\end{algorithm}

\begin{figure}[ht]
\centering
\begin{minipage}{.45\textwidth}
\tikzset{every picture/.style={line width=0.75pt}} 
\begin{tikzpicture}[x=0.75pt,y=0.75pt,yscale=-1,xscale=1, baseline=(XXXX.south) ]
\path (0,223);\path (510,0);\draw     node [anchor=south] (XXXX) {};
\draw  (72,197.61) -- (252.4,197.61)(90.04,62.08) -- (90.04,212.67) (245.4,192.61) -- (252.4,197.61) -- (245.4,202.61) (85.04,69.08) -- (90.04,62.08) -- (95.04,69.08)  ;
\draw    (90,157.75) -- (132.86,132.36) -- (180.07,103.48) ;
\draw    (180.07,103.48) -- (179.73,197.81) ;
\draw    (179.73,197.81) -- (222.59,172.43) -- (234.07,165.15) ;
\draw  [color={rgb, 255:red, 208; green, 2; blue, 27 }  ,draw opacity=1 ][fill={rgb, 255:red, 208; green, 2; blue, 27 }  ,fill opacity=1 ] (107.86,146.05) .. controls (107.86,144.63) and (109.01,143.48) .. (110.43,143.48) .. controls (111.85,143.48) and (113,144.63) .. (113,146.05) .. controls (113,147.47) and (111.85,148.63) .. (110.43,148.63) .. controls (109.01,148.63) and (107.86,147.47) .. (107.86,146.05) -- cycle ;
\draw  [color={rgb, 255:red, 208; green, 2; blue, 27 }  ,draw opacity=1 ][fill={rgb, 255:red, 208; green, 2; blue, 27 }  ,fill opacity=1 ] (127.86,134.05) .. controls (127.86,132.63) and (129.01,131.48) .. (130.43,131.48) .. controls (131.85,131.48) and (133,132.63) .. (133,134.05) .. controls (133,135.47) and (131.85,136.63) .. (130.43,136.63) .. controls (129.01,136.63) and (127.86,135.47) .. (127.86,134.05) -- cycle ;
\draw  [color={rgb, 255:red, 208; green, 2; blue, 27 }  ,draw opacity=1 ][fill={rgb, 255:red, 208; green, 2; blue, 27 }  ,fill opacity=1 ] (147.86,121.72) .. controls (147.86,120.3) and (149.01,119.15) .. (150.43,119.15) .. controls (151.85,119.15) and (153,120.3) .. (153,121.72) .. controls (153,123.14) and (151.85,124.29) .. (150.43,124.29) .. controls (149.01,124.29) and (147.86,123.14) .. (147.86,121.72) -- cycle ;
\draw  [color={rgb, 255:red, 208; green, 2; blue, 27 }  ,draw opacity=1 ][fill={rgb, 255:red, 208; green, 2; blue, 27 }  ,fill opacity=1 ] (167.19,110.05) .. controls (167.19,108.63) and (168.34,107.48) .. (169.76,107.48) .. controls (171.18,107.48) and (172.34,108.63) .. (172.34,110.05) .. controls (172.34,111.47) and (171.18,112.63) .. (169.76,112.63) .. controls (168.34,112.63) and (167.19,111.47) .. (167.19,110.05) -- cycle ;
\draw  [color={rgb, 255:red, 208; green, 2; blue, 27 }  ,draw opacity=1 ][fill={rgb, 255:red, 208; green, 2; blue, 27 }  ,fill opacity=1 ] (187.52,191.72) .. controls (187.52,190.3) and (188.68,189.15) .. (190.1,189.15) .. controls (191.52,189.15) and (192.67,190.3) .. (192.67,191.72) .. controls (192.67,193.14) and (191.52,194.29) .. (190.1,194.29) .. controls (188.68,194.29) and (187.52,193.14) .. (187.52,191.72) -- cycle ;
\draw  [color={rgb, 255:red, 208; green, 2; blue, 27 }  ,draw opacity=1 ][fill={rgb, 255:red, 208; green, 2; blue, 27 }  ,fill opacity=1 ] (207.86,179.72) .. controls (207.86,178.3) and (209.01,177.15) .. (210.43,177.15) .. controls (211.85,177.15) and (213,178.3) .. (213,179.72) .. controls (213,181.14) and (211.85,182.29) .. (210.43,182.29) .. controls (209.01,182.29) and (207.86,181.14) .. (207.86,179.72) -- cycle ;
\draw  [color={rgb, 255:red, 208; green, 2; blue, 27 }  ,draw opacity=1 ][fill={rgb, 255:red, 208; green, 2; blue, 27 }  ,fill opacity=1 ] (228.52,167.72) .. controls (228.52,166.3) and (229.68,165.15) .. (231.1,165.15) .. controls (232.52,165.15) and (233.67,166.3) .. (233.67,167.72) .. controls (233.67,169.14) and (232.52,170.29) .. (231.1,170.29) .. controls (229.68,170.29) and (228.52,169.14) .. (228.52,167.72) -- cycle ;
\draw  [color={rgb, 255:red, 208; green, 2; blue, 27 }  ,draw opacity=1 ][fill={rgb, 255:red, 208; green, 2; blue, 27 }  ,fill opacity=1 ] (183.19,67.05) .. controls (183.19,65.63) and (184.34,64.48) .. (185.76,64.48) .. controls (187.18,64.48) and (188.34,65.63) .. (188.34,67.05) .. controls (188.34,68.47) and (187.18,69.63) .. (185.76,69.63) .. controls (184.34,69.63) and (183.19,68.47) .. (183.19,67.05) -- cycle ;
\draw    (174,46) -- (197,46) ;
\draw (204,41) node [anchor=north west][inner sep=0.75pt]   [align=left] 
{\footnotesize Continuous};
\draw (204,62) node [anchor=north west][inner sep=0.75pt]   [align=left] {\footnotesize Discrete};
\draw (46.67,40.67) node [anchor=north west][inner sep=0.75pt]   [align=left] {{\footnotesize Mempool size}};
\draw (258.67,187.67) node [anchor=north west][inner sep=0.75pt]   [align=left] {$\displaystyle t$};
\draw (176.67,199.67) node [anchor=north west][inner sep=0.75pt]   [align=left] {$\displaystyle t^*$};
\end{tikzpicture}
\caption{Example of a case where the mempool process of transactions with fee density larger than $\phi$, when observed continuously (solid line), hits 0 at time $t^*$ and a transaction with fee density $\phi$ is thus confirmed. However, if we were to observe this process at discrete instances (dots) instead, we would never observe the process hit 0 exactly.}\label{fig:discrete-problem-example}
\end{minipage}
\hfill
\begin{minipage}{.45\textwidth}
\includegraphics[width=\linewidth]{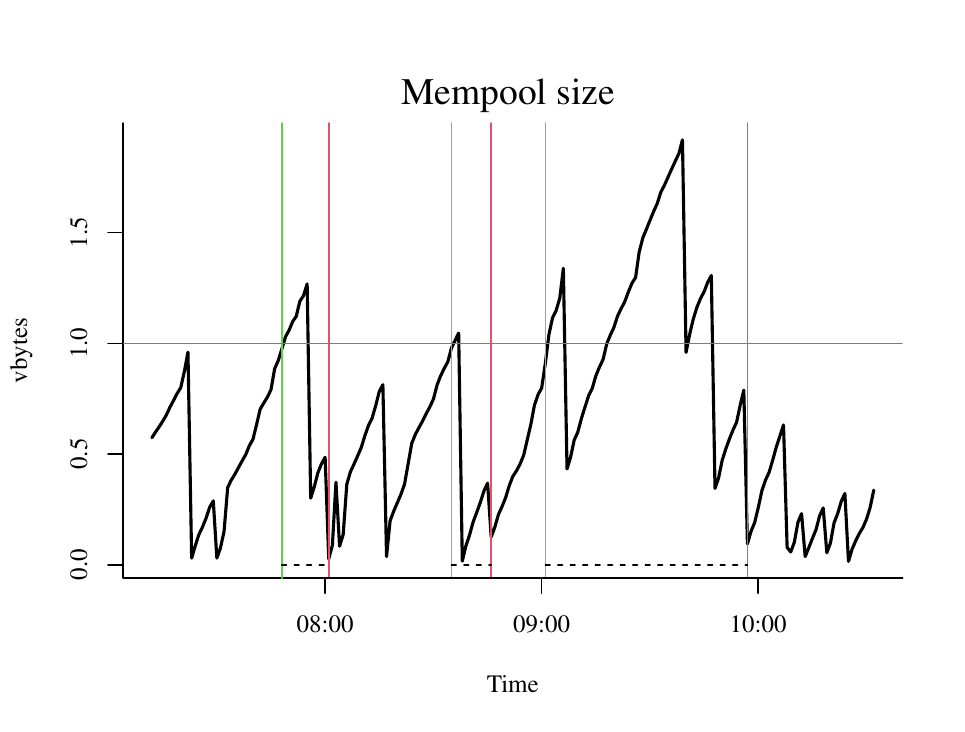}
\caption{
Illustration of how different confirmation times are extracted from the observed mempool process on 03-11-2019. The dark blue line represents $\hat y_\phi=1$, where the confirmation time starts. Green vertical lines denote the initialisation where the process passes through $\hat y_\phi$ and the red lines denote the first time when the process satisfies the  confirmation time criterion in line 3 of Algorithm \ref{alg:getconftime}. The dotted lines indicate the three extracted confirmation times resulting from this procedure. 
}
\label{fig:extraction-example}
\end{minipage}
\end{figure}

\subsection{Details on the model validation}
    \label{sec:cs_validation}

In the following part, for fixed fee density $\phi$, we explain how to extract a confirmation time, say $\hat T_\phi(\hat y_\phi,\hat c_\phi)$, from the mempool data for the same strategically chosen starting position $\hat y_\phi$ and slope $\hat c_\phi$ to obtain as many observations as possible. 
By requiring the extracted confirmation times to have the same parameters, we can present meaningful metrics for comparison, such as a histogram and the mean confirmation time. These can be compared to the model-based performance metrics as a validation measure.
We work with a \textit{global estimator} $\hat c_{\phi}$ for the slope, defined by the sum of all increments divided by the sum of their respective time intervals, where we exclude intervals in which a block has been found:
\begin{equation}
    \label{eq:est_c_full}
    \hat c_{\phi} =  \frac{\sum_{j\geq 1} X^\phi_j\1\{X^\phi_j > 0\}}{\sum_{j\geq 1}(t_j - t_{j-1})\1\{X^\phi_j > 0\} },
\end{equation}
and focus on confirmations of transactions that are sent in the system when the cumulative mempool weight is at median height $m(\cdot)$\footnote{$\hat y_\phi$ is taken fixed, in order to obtain confirmation times that are hypothesised to be of the same distribution. 
While the median is in essence an arbitrary choice, with it, we expect to find the most confirmation times.
},
\begin{equation}
    \label{eq:est_y}
    \hat y_\phi = m \left( \left\{\hat Y_\phi(t_j) \right\}_{j=0}^N\right).
\end{equation}

Concretely, we loop over the data, starting a confirmation time when the cumulative mempool weight process crosses $\hat {y}_\phi$ from below for the first time. 
Then using \texttt{Get\_confirmation\_time}, we obtain the first confirmation time. Repeating this the next time we observe  the cumulative mempool weight process crossing $\hat{y}_\phi$ after the previous confirmation time, we extract non-overlapping confirmation times. 

Note that, over a large time scale, the rate $\hat{c}_\phi$ of incoming transactions with fee density larger than $\phi$ may change. A change in the rate will result in confirmation times not necessarily following the hypothesized distribution $T(\hat y_\phi,\hat c_\phi)$. To safeguard ourselves against this, after extracting the confirmation times from the data according to the criterion detailed above, we consider only the ones for which a \textit{local} estimate of the rate of incoming transactions with fee density larger than $\phi$ is sufficiently close to the global estimator $\hat c_\phi$. For the $i$-th confirmation time that starts at $t_{j_1}$ and ends at $t_{j_2}$, we estimate the rate locally as:
\begin{equation}
    \label{eq:est_c_on_data}
    \hat c^*_{\phi,i} = \frac{ \sum_{j=j_1+1}^{j_2} X^\phi_j\1\{X^\phi_j > 0\}}{ \sum_{j=j_1+1}^{j_2}(t_j - t_{j-1})\1\{X^\phi_j > 0\} }.
\end{equation}
In our results, we only consider observations where local estimates differ at most $\delta=0.05$ from the global estimate $\hat c_\phi$. This procedure is detailed in 
Algorithm \ref {alg:extract_conf_validation}.

\begin{minipage}{0.47\textwidth}
\begin{algorithm}[H]
\centering 
\caption{Extract validation sample for fixed $\phi$}
    \begin{algorithmic}[1]
        \State \textbf{Input}: $\{Y_\phi(t_j)\}_{j=0}^N$
        \State  Estimate $\hat c_\phi $ by Equation \eqref{eq:est_c_full}
        \State Estimate $\hat y_\phi$ by Equation \eqref{eq:est_y}
        \State $j=0$
        \While{$j < N$}
        \State $j_0=\inf\{i>j: Y_\phi(t_{i})<\hat y_\phi \wedge Y_\phi(t_{i+1})>\hat y_\phi\}$
        \State $T(\hat y_\phi, \hat c_\phi) =$ \texttt{G}($\{Y(t_k)\}_{k=j_0}^N),\hat c_\phi$)
        \State Estimate $\hat c_{\phi}^*$ by \eqref{eq:est_c_on_data} on $\{Y_\phi(t_k)\}_{k=j_0}^{j_0+T(\hat y_\phi, \hat c_\phi)}$
        \If{ $\hat c_{\phi}^* \in (\hat c_\phi-\delta, \hat c_\phi+\delta)$}
        \State Keep $T(\hat y_\phi, \hat c_\phi)$ in confirmation times
        \Else{ 
        Dispose $T(\hat y_\phi, \hat c_\phi)$}
        \EndIf 
        \State $j=j+ T(\hat y_\phi, \hat c_\phi)$
        \EndWhile
        \State \textbf{return} confirmation times
    \end{algorithmic}
    \label{alg:extract_conf_validation}
\end{algorithm}
\vspace{0.1cm}
\end{minipage}
\hfill
\begin{minipage}{0.45\textwidth}
    \begin{algorithm}[H]
    \centering
    \caption{Extract optimal bucket for fixed $t_i$}
\begin{algorithmic}[1]
  \State  \textbf{Input}:
    $\{(Y_\phi(t_j)\}_{j=j_i-2500}^{N}$
    \For{$\phi \in \{\phi_1,\ldots\phi_n$\}}
    \For{$i$ in $1:7500$}
    \State Sample $t_0$ uniformly between $[j_i-2500,j_i]$ 
    \State Save  $T_\phi =$ 
    \texttt{G}($\{(Y_\phi(t_j)\}_{j=j_i-2500}^{j_i}$,0)
    \EndFor 
    \State Determine $\hat F_\phi$ based on $\{T_{\phi,j}\}_{j=1}^{7500}$
    \State Determine $\hat c$  based on \eqref{eq:c_est_pred} and $\{T_{\phi,j}\}_{j=1}^{7500}$
    \State Determine $\tilde F$
    \EndFor
    \State Determine $\hat b$ based on Equation \eqref{eq:hatb}
    \State Determine $\tilde b$ based on
    Equation \eqref{eq:tildeb}
    \State Determine $b^*(t_j)$ based on Equation \eqref{eq:bstar}
    \State \textbf{return} $\hat b$, $\tilde b$
    and $b^*$
\end{algorithmic}
\label{alg:scores}
\end{algorithm}
\vspace{0.1cm}
\end{minipage}

From Table \ref{tab:CS1}, we observe that, in general, model-based estimations approximate the real-world transaction data decently well in terms of the mean. Deviations are mainly due to under-estimations by the model-based approach, implying the possibility of an unknown underlying mechanism that was not accounted for. However, in the higher cost buckets, such deviations cause an error of only a couple of minutes, while for lower cost buckets, estimations tend to fall within the confidence limits of the estimated mean. This implies that the model-based method is usable in practice.

\begin{table}
\centering
\begin{tabular}{c c c c c c  c}
  \hline
$\phi$ & $\hat\mu\in(LCL,UCL)$ & $\mu=\mathbb{E}[T_y]$& $\tilde{\mu}$ (BM)& $n$ & $\hat y_\phi$ & $\hat c_\phi$ \\ 
  \hline
  2 &$ 100.1 \in ( 66.16 , 185.4 )$ & 83.22&   82.95     & 191 & 4.35 & 0.89 \\ 
  3 &$ 29.56 \in ( 22.36 , 56.39 )$ & 19.73 &19.69       & 368 & 2.36 & 0.81 \\ 
  4 &$ 16.98 \in ( 13.80 , 23.80 ) $&11.76 & 11.72       & 527 & 1.74 & 0.77 \\ 
  5 &$ 7.93 \in ( 6.91 , 9.63 )$ & 7.20&    7.17         & 540 & 1.30 & 0.72 \\ 
  6 &$ 6.80 \in ( 5.91 , 8.30 ) $& 5.72 &   5.70           & 640 & 1.12 & 0.69 \\ 
  7 &$ 5.52 \in ( 4.87 , 6.66 )$ & 5.14&    5.11           & 849 & 1.00 & 0.68 \\ 
  8 & $4.82 \in ( 4.36 , 5.50 ) $& 4.42&   4.40           & 1007 & 0.91 & 0.66 \\ 
  10 & $4.37 \in ( 3.98 , 4.94 )$ & 3.99&  3.98        & 1174 & 0.85 & 0.65 \\ 
  12 & $3.53 \in ( 3.25 , 3.93 )$ & 3.35&  3.33       & 1519 & 0.75 & 0.63 \\ 
  14 & $2.97 \in ( 2.76 , 3.27 )$ & 2.67& 2.66      & 1998 & 0.65 & 0.59 \\ 
  17 & $2.69 \in ( 2.54 , 2.90 ) $& 2.36&  2.34       & 2398 & 0.58 & 0.57 \\ 
  20 & $2.24 \in ( 2.14 , 2.37 )$ & 2.01&  2.00        & 3005 & 0.49 & 0.54 \\ 
  25 & $1.98 \in ( 1.89 , 2.08 )$ & 1.71& 1.70            & 3578 & 0.41 & 0.50 \\ 
  30 & $1.63 \in ( 1.57 , 1.70 ) $& 1.41& 1.40        & 4242 & 0.30 & 0.44 \\ 
  40 & $1.48 \in ( 1.42 , 1.54 ) $& 1.27&  1.26       & 4425 & 0.23 & 0.40 \\ 
  50 & $1.27 \in ( 1.23 , 1.32 ) $& 1.14&  1.14        & 4425 & 0.16 & 0.34 \\ 
  60 & $1.20 \in ( 1.16 , 1.24 ) $& 1.09&  1.08       & 4374 & 0.12 & 0.30 \\ 
  70 & $1.12 \in ( 1.08 , 1.15 )$ & 1.05&  1.04       & 4367 & 0.08 & 0.26 \\ 
   \hline
\end{tabular}
	\caption{The column $\phi$ captures the various fee densities in sat/vbyte. The column $\hat{\mu}$ captures the sample mean in 10 minutes of the Bitcoin  confirmation time data for various values of $\phi$. Included are the lower and upper confidence limits (LCL and UCL, respectively). The confidence limits are computed as nonparametric bootstrap confidence intervals using Efron's bootstrap (see \cite[Section 14.3]{Efron1994}) with $10^5$ bootstrap samples. The last three columns present the number of extracted confirmation times ($n$), the starting position ($\hat y_\phi$ in vMB) and the slope ($\hat c_\phi$ in vMB per 10 minutes). The column $\mu=\mathbb{E}[T_y]$ and $\tilde \mu(BM)$ capture, for the purpose  of comparison, the theoretical expectation of the  confirmation time of the CL and the diffusion approximation respectively with block mean inter-arrival time of 1.010406*10 minutes and block size of $B=0.956$ vMB.}
	\label{tab:CS1}
 \end{table}

\subsection{Details on the model performance}
\label{sec:cs_performance}
Theorem \ref{thm:diffusion_limit}
showed that in addition to the CL, the BM can be used to determine the confirmation times. This was also validated in Section \ref{sec:cs_validation}. 
In the following part we show that it also does the prediction better than a data-driven approach that closely resembles \texttt{estimatesmartfee}, an established fee estimator.

This data-driven method collects \textit{transaction confirmation time data} for different buckets $\phi_1,...,\phi_n$ to construct empirical cdf-s $\hat F_{\phi_1},...,\hat F_{\phi_n}$. Then, for a given maximum time $t^*$, it determines the optimal bucket $\hat b$ as the cheapest bucket so that confirmation before time $t^*$ occurs with probability at least $0.95$, i.e.:
\begin{equation}
\label{eq:hatb}
 \hat b = \min\big\{\phi \in \{\phi_1,...,\phi_n\}: 1-\hat F_\phi(t^*)<0.05\big\}.
\end{equation}
Our method works similarly but instead uses the inverse Gaussian distribution $\tilde F_{c(\phi),y(\phi)}=\tilde F_\phi$. Then for maximum time $t^*$ and $y$ that is observed from the data it determines the optimal bucket $\tilde b$ by
\begin{equation}
\label{eq:tildeb}
 \tilde b = \min\big\{\phi \in \{\phi_1,...,\phi_n\}: 1-\tilde F_\phi(t^*)<0.05\big\}.
\end{equation}

It remains to show how $c$ is determined.
As the $\hat b$ is determined by only transaction data $\{T_1,...,T_n\}$, we also only use transaction data for a fair comparison.
For each bucket, we use maximum likelihood estimators for the parameters  of the inverse Gaussian \cite[Equation 8]{Folks1978} to find
\begin{equation}
    \frac{ y+\expec[S_y]}{1-c} \approx 
    \frac{1}{n}\sum_{i=1}^n T_i,
\end{equation}
and
\begin{equation}
    (y+\expec[S_y])^2 \approx \left( 
    \frac{1}{n}\sum_{i=1}^n \left(\frac{1}{T_i}-\frac{1}{ \frac{1}{n}\sum_{i=1}^n T_i}\right)
    \right)^{-1}.
\end{equation}
This can be solved for $c$ to give the estimator 
\begin{equation}
\label{eq:c_est_pred}
    \hat c = 1- \frac{ \left( \frac{1}{n}\sum_{i=1}^n \left(\frac{1}{T_i}-\frac{1}{ \frac{1}{n}\sum_{i=1}^n T_i}\right)\right)^{-1/2} }{  \frac{1}{n}\sum_{i=1}^n T_i }.
\end{equation}
We next describe how to compare $\hat b$ and $\tilde b$.
In the data we consider 70 time instances $(t_{j_1},...,t_{j_{70}})$ with on average a week between sequential time instances. For each instance $i$ we define the \textit{test-period}, as the period of approximately 42 hours before $t_{j_i}$ until $t_{j_i}$. 
During this test period, for every bucket $j$, we extract $N$ confirmation times, which we do as follows. First, draw a point uniformly in the test-period and define the confirmation time based on \texttt{Get\_confirmation\_time} as described in Algorithm \ref{alg:getconftime}, with $\varepsilon=0$. 
For every bucket, we repeat this $N$ times to obtain a sample $(T^{(j_i)}_k)_{k=1}^N$. Then, for a given maximum time $t^*$, we determine the optimal bucket $b^*$ by 
\begin{equation}
\label{eq:bstar}
b^*(t_{j_i}) = \min\big\{\phi \in \{\phi_1,...,\phi_n\}: \texttt{Get\_confirmation\_time}(\{Y_\phi(t_j)\}_{j=j_i}^N,0) <t^*  \big\}.
\end{equation}

Algorithm \ref{alg:scores} gives an overview how the optimal buckets can be derived from the data for a fixed time instance. 
In Figure \ref{fig:scores_mb}, \ref{fig:scores_dd} and \ref{fig:box_plot_scores} we plot the \textit{scores} of the model-based ($s_{M}$) approach and the data-driven ($s_{D}$) approach. The scores are defined by 
\begin{equation}
s_{M}(t_{j_i}) = \tilde{b} - b^*(t_{j_i})
\qquad \text{ and } \qquad 
s_{D}(t_{j_i}) = \hat{b} - b^*(t_{j_i}).
\end{equation}
Thus a score of 0 implies that the predicted optimal bucket coincides with the true optimal bucket, a positive score implies an \textit{overpay} of fee, and a negative score implies that the transaction confirmation time in practice had exceeded $t^*$, and therefore was confirmed too \textit{late}. In Table \ref{tab:CS2} we provide details on the number of transactions that were too late and were overpaid for (and to what degree).

Based on Figures \ref{fig:scores_mb}--\ref{fig:box_plot_scores} we observe that our model-based predictions perform stronger in general than the data-driven approach. Only in the number of \textit{late} transactions, the data-driven approach outperforms the model-based approach, but at the cost of overpaying on many transactions by a significant amount.
Based on this, the model-based method is recommended unless the user is willing to risk overpaying to mitigate a possible late transaction.

\begin{table}[ht]
    \centering
    \begin{tabular}{ccc}
         & Model-based & Data-driven \\
         \hline
       \% Optimal &46& 11\\
       \% Late (mean time late, given late)& 21 (9.0) & 4 (90.1)\\
       \% Overpay (mean overpay, given overpay)& 33 (39.2) & 84 (31.8)
    \end{tabular}
    \caption{Summary statistics on the performance of the model-based and data-driven methods for 70 instances. The mean time late of the data-driven method is heavily influenced by an outlier of value 259.}
    \label{tab:CS2}
\end{table}

\begin{figure}[ht]
\centering
\begin{minipage}{.45\textwidth}
\includegraphics[width=\linewidth]{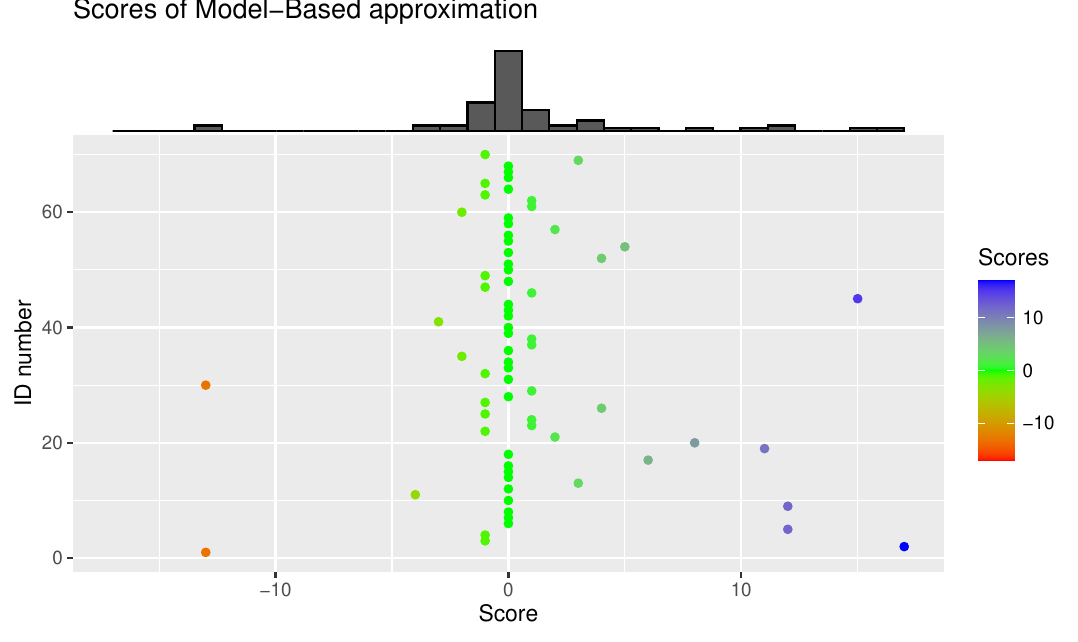}
\caption{Scores of the model-based approach ($s_M(t_{j_i})$). The $y$-axis denotes the ID number ($i$) of the 70 times instances and the $x$-axis denotes the score. Here, $t^*=5(\times 10m)$.}\label{fig:scores_mb}
\end{minipage}
\hfill
\begin{minipage}{.45\textwidth}
\includegraphics[width=\linewidth]{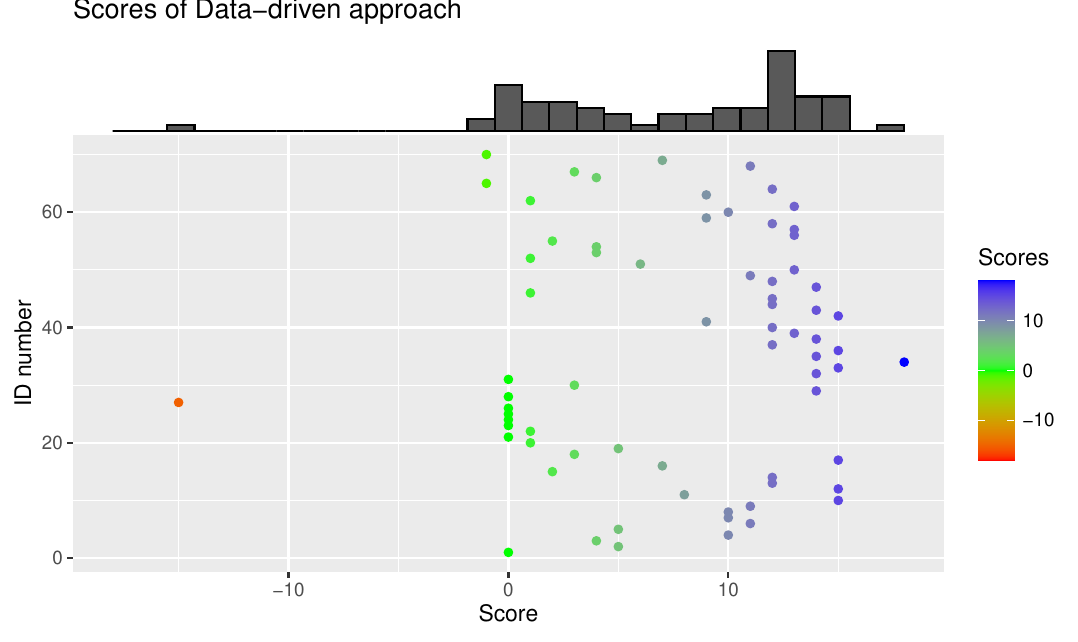}
\caption{Scores of the data-driven approach ($s_D(t_{j_i})$). The $y$-axis denotes the ID number ($i$) of the 70 times instances and the $x$-axis denotes the score. Here, $t^*=5(\times 10m)$.}\label{fig:scores_dd}
\end{minipage}
\\
\begin{minipage}{.67\textwidth}
\includegraphics[width=\linewidth]{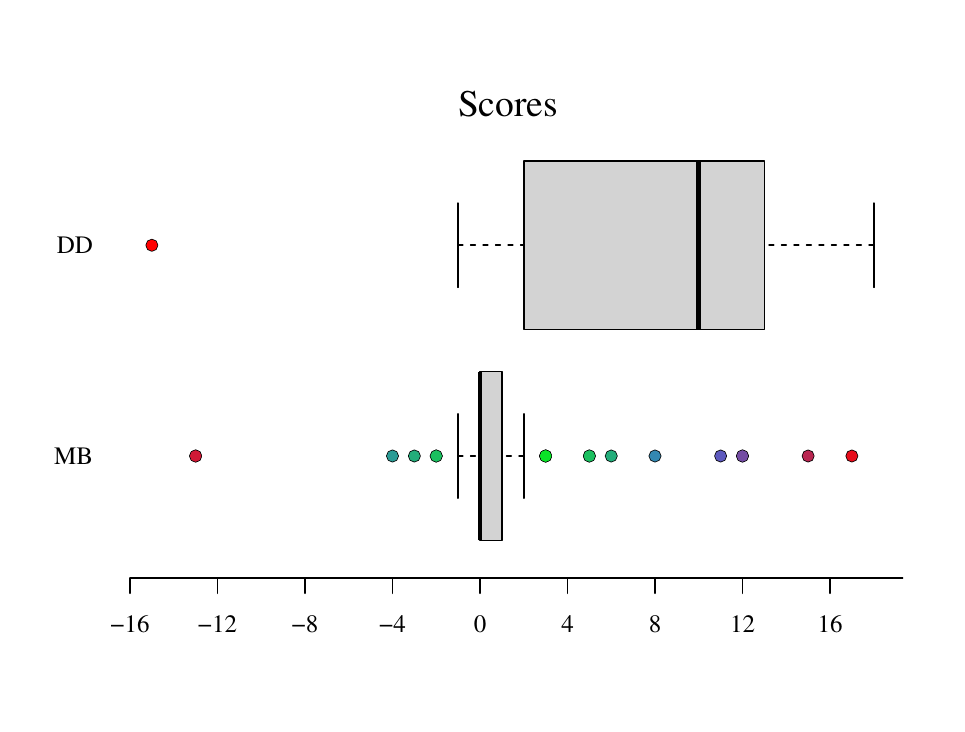}
\caption{Box plot of the scores of the data-driven approach  (top) and the model-based approach  (bottom). }\label{fig:box_plot_scores}
\end{minipage}
\end{figure}

\section{Conclusion}
\label{sec:conclusion}
In this work, we study confirmation times in proof-of-work based cryptocurrencies (such as Bitcoin) via a holistic model-based approach for a ready-to-use optimal fee prediction method. Our work builds on three niches: 

Firstly, we consider a theoretical model (stochastic process) that captures the cumulative mempool weight. We consider the same stochastic process as in \cite{Gundlach2021}: The cumulative mempool weight process is modelled by a stochastic process that grows linearly over time and at random instances (corresponding to  the mining of a block) it jumps down by a fixed amount (the amount corresponds to the size of a block). This  is a special case of the so-called Cram\'er-Lundberg model, a model fundamental to actuarial sciences. Here, a confirmation time of a Bitcoin transaction coincides with the time required for the stochastic process to hit 0 for the first  time (in the actuarial  science jargon, this  is the  so-called time to ruin). 
We derive explicit expressions for the confirmation time distribution, both in minutes and in the number of blocks. We also give an explicit expression for the mean confirmation time and demonstrate that this quantity (as a function of the initial mempool weight / starting position, say $y$ and the rate transactions accumulate over time, say $c$) can be closely approximated by a linear function of $y$ and $c$, c.f. Equation \eqref{eq:E(t)approx}, which is trivially fast to evaluate.

Secondly, we relate the batch-service queue (BSQ), a model that has been proposed for Bitcoin confirmation times in the literature, to the Cram\'er-Lundberg (CL) model via a fluid limit and to the Brownian motion (BM) via a diffusion limit.
We show, under a feasible scaling of the BSQ, that, whether one considers the transaction count or the cumulative weight, the confirmation times of the BSQ converge to the confirmation times of the CL model and BM in fluid and diffusion limiting regime respectively. We also argue that the necessary scaling is realistic in practice. Our model fits therefore well within the state-of-the-art stochastic models for Bitcoin confirmation times.

Thirdly, we validate and compare the proposed model to confirmation times extracted from real-world Bitcoin data. We compare the empirical distribution and the sample mean of the extracted  confirmation times to those obtained from the model,  and demonstrate that they are relatively close. Furthermore, we show that our model outperforms purely data-driven approaches. Therefore, a user can, after observing the model parameters, predict its confirmation time (distribution and mean) and use this to determine an optimal fee.

In this work, we  focus the analysis on Bitcoin. However, the analysis extends to other pay-for-speed applications where the many small jobs or requests arrive in a short time span and are served in large batches. Even in instances where there may be a certain dependence of the arrival and servings of jobs, our models provides intuitive and easy methods to determine the optimal pay for different speed benchmarks.

\paragraph{Acknowledgements}

The authors thank Onno Boxma for his feedback and suggestions. RG and SK are supported by the Netherlands Organisation for Scientific Research (NWO) through Gravitation-grant NETWORKS-024.002.003.

\newpage

\bibliographystyle{apalike}
\bibliography{Zbib}

\begin{appendices}

\section{Alternative proof for the tail distribution of the number of blocks to confirmation}
\label{app:Alternative proof Exp conf time}
In the following, we provide an alternative proof to Lemma \ref{prop:ET}. Here, we link the number of blocks to confirmation to the number of customers in a busy period in a dual D/M/1 queue.

Let the time between arrivals of blocks be denoted by $\{A_i\}_{i\geq0}$.
In the CL model, between two consecutive block minings, the mempool size has increased by $cA_i$ and it has decreased by $1$. The former is over a continuous period of time (continuous creeping of the process), while the latter is due to a jump. However, it is possible to interchange these to obtain a process where jumps occur every $1$ time unit and the jumps have a stochastic size of $cA_i$, which is again exponentially distributed. This new process, in between jumps, decreases linearly with slope $1$ and at a jump it increases by $cA_i$. This switched process then models the workload of a D/M/1 queue. We refer to Figure \ref{fig:Queueingduality} for a visual comparison.

Denote with $Q_y$ the number of arrivals within a busy period of a D/M/1 queue (with the aforementioned arrival and service rates) that starts with an initial amount of work $y$ and the arrival of the first customer at $t=0$. We next formalise the relation of the number of customers in the busy period to the number of blocks to confirmation as described above: 
\begin{figure}[hbt]
    \centering
    \includegraphics[scale=0.7]{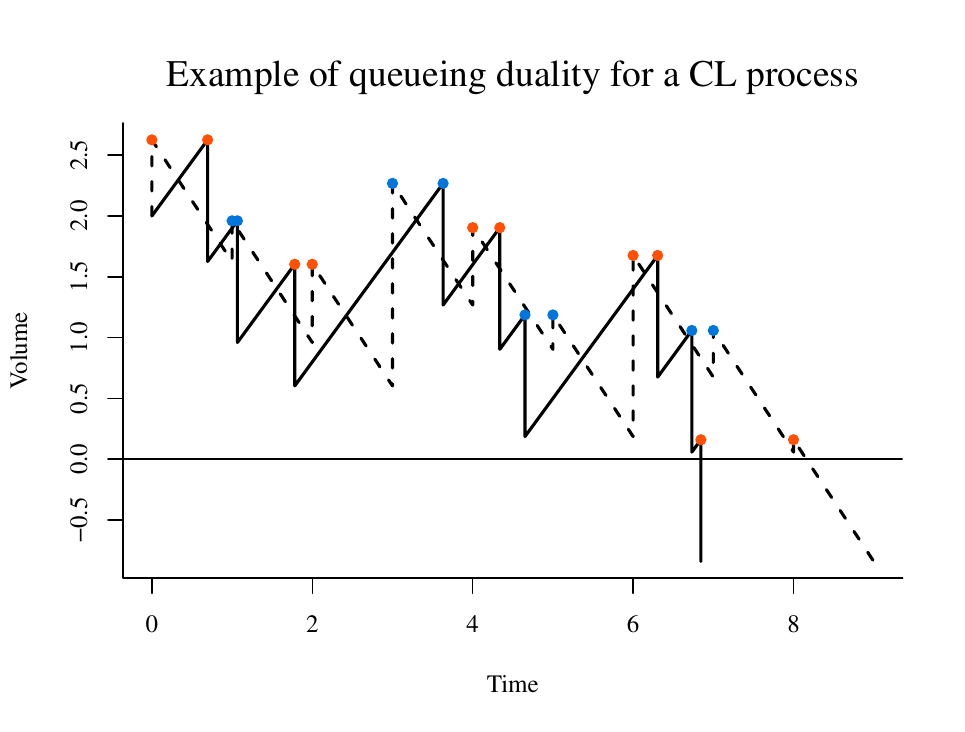}
    \caption{Example of the queueing duality for the CL model. Note that all peaks occur at the same height, but at different points in time (indicated by matching colours). Furthermore, note that the number of peaks before the $x-$axis is hit for the respective processes is the same. The dashed line then corresponds to the workload of a D/M/1 queue. }
    \label{fig:Queueingduality}
\end{figure}

\begin{lemma}[Queuing duality]
\label{lem:Q=N}
Consider a $D/M/1$ queue where customers arrive (deterministically) every 1 time unit and service requirements are exponentially distributed with mean $c$.
Let $Q_y$ be the number of arrivals within a busy period that started with the arrival of a customer at time $t=0$ and with  an initial workload of $y$. Then $N_y\overset{d}{=}Q_y$.
\end{lemma}
\begin{proof}
   Note that both $N_y$ and $Q_y$ are stopping times on stochastic processes driven by i.i.d. inter-arrival times $\{A_i\}_{i\geq1}$ and service times $\{cA_i\}_{i\geq1}$ respectively, where the random variables $A_i$ are exponentially distributed with mean 1.
In fact, both $N_y$ and $Q_y$ are defined by
   \begin{equation}
       N_y = Q_y =
       \inf\left\{k\geq 1: y+c\sum_{i=1}^k A_i< k\right\},
   \end{equation}
which implies that $N_y$ and $Q_y$ have the same distribution.\end{proof}

A relevant theorem, for the specific case $y\in\mathbb{N}$, for the tail distribution of the number of customers in the busy period of the dual $D/M/1$ can be found in the queueing literature. We present this in Theorem \ref{thm:Stadje}:
\begin{theorem}[Number of customers in a D/M/1 busy period]
\label{thm:Stadje}
Consider a D/M/1 queue where customers arrive at fixed time intervals $(a_i)_{i\geq 0}$ such that $0=a_0<a_1<...$. Furthermore, suppose the queue works at unit speed and has service rate $\mu$. If at the start of a busy period $m$ customers are present, then
\begin{equation}
    \prob[ Q_y>n ] =
    \sum_{k\in K_n(m)}
    \e^{-\mu a_n}
     \frac{ \mu^{k_0+...+k_{n-1}} }{k_0!\cdots k_{n-1}!}\prod_{i=0}^{n-1} (a_i-a_{i-1})^{k_i},
\end{equation}
where $K_n(m)=\{ \vec{k}\in \mathbb{N}^n \mid k_0+...+k_i\leq i+m\}$.
\end{theorem}
\begin{proof}
   See \cite[Section 2]{Stadje1995}.
\end{proof}

The proof of Proposition \ref{prop:DM1BP} is a slight adaptation of the result of Theorem \ref{thm:Stadje} where there is some additional  work $\varepsilon \in (0,1)$ in the system at time $0$ (or alternatively formulated, the first customer has some service time $\varepsilon+cA_1$, while the other customers have service times $cA_2, cA_3, \ldots$).

\section{Tightness}
\label{app:tightness}

We use the Arzel\'a-Ascoli characterisation, see for example \cite[13.2]{Bingham1987}. This characterisation states tightness for $x_n(t)\in D[0,\tau]$ follows from the following properties:
\begin{enumerate}
    \item Almost surely bounded:
    \begin{equation}
    \label{eq:AAcriterion as bounded}
        \lim_{M\to\infty }
        \limsup_{n\to\infty}
        \prob\Big(\sup_{t\in[0,\tau]} |x_n(t)|>M\Big)
        =0;
    \end{equation}
    \item Modulus of continuity: Let $\Pi_\tau(\delta)$ be the set of partitions $\vec{s}=(s_0,\ldots,s_k)$ with $k\in \mathbb{N}$, $s_0=0$ and $s_k=\tau$ of $[0,\tau]$ such that each sub-interval is of length at least $\delta$. Then
    for given $\varepsilon>0$
    \begin{equation}
    \label{eq:AAcriterion modulus c}
        \lim_{\delta\to 0}
        \limsup_{n\to\infty}
        \prob\Big( \inf_{\vec{s}\in \Pi_\tau(\delta)}
        \max_{i\geq 1}\sup_{s_{i-1}\leq t_1<t_2\leq s_{i}}|x_n(t_1)-x_n(t_2)|>\varepsilon\Big)=0.
    \end{equation}
\end{enumerate}
In the following we prove tightness based on Equation \eqref{eq:AAcriterion as bounded} and Equation \eqref{eq:AAcriterion modulus c}.

    Let $\tau$ be given. Then we start by showing that $Q^{(n)}_{\nu,\lambda,K}(t)/(nK)$ is almost surely bounded. Define a process 
    \begin{equation}
        R^{(n)}_{\nu}(t) = \sum^{m(n)+A_{\nu n}(t)}_{i=1}X_i,
    \end{equation}
    i.e. the BSQ where no departures take place. Then clearly for all $t\geq 0$, $Q^{(n)}_{\nu,\lambda,K}(t)/(nK)$ is stochastically dominated by $R^{(n)}_{\nu,\lambda,K}(t)/(nK)$. Fix $\varepsilon>0$ and $M>0$, then by \eqref{eq:fluid_initial_conditions} we find there exists a $\delta>0$ and $n_0>0$, such that by the Markov inequality for $n>n_0$
    \begin{equation}
       \prob\Big(
        \frac{1}{Kn}
        Q^{(n)}_{\nu,\lambda,K}(t)>M
        \Big)\leq 
        \prob\Big(
        \frac{1}{Kn}
        R^{(n)}_{\nu,\lambda,K}(t)>M
        \Big)
        \leq 
        \frac{Ky+\expec[X]\tau\nu}{KM}(1+\delta).
    \end{equation}
    Now setting $M>(1+\delta)(Ky+\nu\tau\expec[X])/(K\varepsilon)$ provides a uniform bound in $t$, thereby showing the first criterion for tightness of Equation \eqref{eq:AAcriterion as bounded}.\\
    \medskip
    We next show the second criterion. We aim to construct a partition $\vec{\pi}^*$ such that fluctuations within sub-intervals are small.
    Let $\varepsilon>0$ be given, then for $\delta>0$ to be specified later, we construct $\vec{\pi}^*$ as follows:
    we put partition bounds on all block mining events (which we recall coincide with batch service points), then within these partition bounds we put sub-intervals of length at exactly $\delta$ if possible. This means that between consecutive block minings, there are sub-intervals of length exactly $\delta$ and one sub-interval that is at most $2\delta$. 
    We need to show that $\vec{\pi}^*\in\Pi_\tau(\delta)$ and that Equation \eqref{eq:AAcriterion modulus c} holds, which we do as follows.

    \paragraph{$\vec{\pi}^*$ is an appropriate partition:}
    It is sufficient to show that the time between consecutive blocks is more than $\delta$ w.h.p. We use that for a Poisson process, conditioned on the number of arrivals in an interval, the arrivals occur uniformly within that interval. Let $\eta>0$ be given and
 take $\delta<\delta_1=\eta/(2\lambda\tau)$, then
    \begin{equation}
    \label{eq:pi*in Pi delta}
    \begin{aligned}
        \prob(\vec{\pi}^* \not\in \Pi_\tau(\delta)) &=
        \sum_{i=0}^\infty 
        \prob(\vec{\pi}^* \not\in \Pi_\tau(\delta)\mid D_\lambda(\tau) = i)\prob(D_\lambda(\tau) = i)\\&=
        \sum_{i=0}^\infty \prob(\text{$\exists$ a pair of uniforms closer than $\delta$ in $[0,\tau]$}\mid D_\lambda(\tau) = i)\prob(D_\lambda(\tau) = i)\\&
        \leq
         \sum_{i=0}^\infty
         (1-(1-2i\delta)) \prob(D_\lambda(\tau) = i)
         \\&=2\delta\lambda\tau 
         <\eta.
    \end{aligned}
    \end{equation}
 \paragraph{No large fluctuations within partition intervals}   
 Fix $t\in[0,\tau]$, then we first consider fluctuations in $[t,t+\delta]$. 
  Based on the construction of $\vec{\pi}^*$ there is no block arrival in this interval, so that it suffices to show that for given $\varepsilon>0$, we can take $\delta$ small enough such that 
 \begin{equation}
 \label{eq:AA modulus as bounded}
     \prob\Big(\frac{1}{Kn}\sum_{i=1}^{A_{\nu n}(\delta)}
     X_i>\varepsilon\Big)\to 0.
 \end{equation}
 However, the exact speed at which this probability converges to 0 is of importance.
 Let $\delta<\delta_2=K\varepsilon/\expec[X]$and $0<\xi < K\varepsilon/(\nu\expec[X])-\delta $. Then
conditioning on $A_{\nu n}(\delta)$ being large, we find 
 \begin{equation}
 \begin{aligned}
     \prob\Big(\frac{1}{Kn}\sum_{i=1}^{A_{\nu n}(\delta)}
     X_i>\varepsilon\Big)
    &\leq \prob\Big(\frac{1}{Kn}\sum_{i=1}^{n\nu(\delta  + \xi)}
     X_i>\varepsilon\Big)+
     \prob(A_{\nu n}(\delta) >n\nu(\delta  + \xi)).
\end{aligned}
\end{equation}
Set $\sigma^2= \var[X]$, then
by the central limit theorem for the first term, we find 
\begin{equation}
\begin{aligned}
\prob\Big(\frac{1}{Kn}\sum_{i=1}^{n\nu(\delta  + \xi)}
     X_i>\varepsilon\Big) &= \prob\bigg(\frac{\sum_{i=1}^{n\nu(\delta+\xi)}
     X_i - n\nu(\delta+\xi)\expec[X]}{\sqrt{2\pi\sigma^2n\nu(\delta+\xi)}}>\frac{\varepsilon Kn -n\nu(\delta+\xi)\expec[X]}{\sqrt{2\pi\sigma^2n\nu(\delta+\xi)}} \bigg)
    \\&=\prob\bigg(N(0,1) > \sqrt{n}\frac{\varepsilon K - \nu(\delta+\xi)\expec[X]}{\sqrt{2\pi\sigma ^2\nu(\delta+\xi)}}
    \bigg)(1+o(1))
    \\&=
    C\exp\bigg\{-n\Big(\frac{\varepsilon K - \nu(\delta+\xi)\expec[X]}{\sqrt{2\pi\sigma ^2\nu(\delta+\xi)}}
    \Big)^2\bigg\}(1+o(1)).
\end{aligned}
 \end{equation}
Via a similar argument, we find also
\begin{equation}
    \prob(A_{\nu n}(\delta) >n\nu(\delta  + \xi))
    = C\exp\bigg\{-\nu n\Big(\frac{\xi}{\sqrt{2\pi\delta}}\Big)^2\bigg\}(1+o(1)).
\end{equation}
This implies that 
\begin{equation}
    \prob\Big(\frac{1}{Kn}\sum_{i=1}^{A_{\nu n}(\delta)}
     X_i>\varepsilon\Big)
     \leq
     2C\e^{-c(\delta)n}(1+o(1)),
\end{equation}
where we recall that $c(\delta)>0$ for $\delta<K\varepsilon/(\nu\expec[X])$.

\paragraph{Combining the results:}
Based on Equation \eqref{eq:pi*in Pi delta} and based on the fact that $(nK)^{-1}Q^{(n)}_{\nu,\lambda,K}(t)$ is increasing within sub-intervals of $\vec{\pi}^*=(s^*_i)_{i\geq 1}$.
We find for $\delta<\min\{\delta_1,\delta_2\}$
\begin{equation}
    \begin{aligned}
       &  \prob\Big( \inf_{\vec{s}\in \Pi_\tau(\delta)}
        \max_{i\geq 1}\sup_{s_{i-1}\leq t_1<t_2\leq s_i}|(nK)^{-1}Q^{(n)}_{\nu,\lambda,K}(t_1)-(nK)^{-1}Q^{(n)}_{\nu,\lambda,K}(t_2)|>\varepsilon\Big)
       \\&\qquad \leq 
        \varepsilon +\prob\Big(\max_{i\geq 1}
        |(nK)^{-1}Q^{(n)}_{\nu,\lambda,K}(s^*_{i-1})-(nK)^{-1}Q^{(n)}_{\nu,\lambda,K}(s^*_i)|>\varepsilon\Big)
        \\&\qquad\leq 
        \varepsilon +(1-(1-2C\e^{-c(\delta)n })^{1/\delta})(1+o(1)).
    \end{aligned}
\end{equation}
    Taking $n\to \infty$ followed by $\delta \to 0$ shows that Equation \eqref{eq:AAcriterion modulus c} is satisfied.
\end{appendices}

\end{document}